\documentclass[a4paper, 11pt]{article}

\usepackage{framed}
\usepackage{amssymb}
\usepackage{latexsym}
\usepackage{amsfonts}
\usepackage{amsthm}
\usepackage{amsbsy}
\usepackage{amsmath}
\usepackage{graphics,graphicx}
\usepackage{multirow,multicol}
\usepackage{bm}
\usepackage{bbm}
\usepackage{color}
\usepackage{float}
\usepackage{tikz}
\usepackage{mathrsfs}
\usepackage{epstopdf}
\usepackage{algorithm}
\usepackage{algorithmicx}
\usepackage{algpseudocode}
\usepackage{subcaption} 
\usepackage{enumitem}
\usepackage[left=2.5cm,right=2.5cm,top=2cm,bottom=2cm]{geometry}

\newtheorem{theorem}{Theorem}
\newtheorem{lemma}{Lemma}

\newtheorem{remark}{Remark}

\newcommand{\mg}{\mathcal{G}}

\newcommand{\bbZ}{\mathbb{Z}}
\newcommand{\bbf}{\boldsymbol{f}}
\newcommand{\bbR}{\mathbb{R}}

\begin{document}

\title{A Globally Conservative Compact Framework for Conservation Laws: Fourth-Order Schemes with Enhanced Resolution and Stability}

\author{
	Weifeng Hou \thanks{School of Mathematics and Computational Science, XiangTan University, Xiangtan, Hunan, China, email: {\tt houwf@smail.xtu.edu.cn}.},~~
	Zhangpeng Sun \thanks{School of Mathematics and Computational Science, XiangTan University, Xiangtan, Hunan, China, email: {\tt zpsun@xtu.edu.cn}.},~~
	Wenqi Yao \thanks{School of Mathematics, South China University of Technology, Guangzhou, Guangdong, China, email: {\tt yaowq@scut.edu.cn}.},~~
	Liupeng Wang \thanks{School of Computational Science and Electronics, Hunan Institute of Engineering, Xiangtan, Hunan, China, email: {\tt wangliuping@163.com}.}
}

\maketitle

\begin{abstract}
The compact finite difference method is a powerful tool for discretizing conservation laws, owing to its inherent flexibility in developing high-resolution and highly stable schemes. 
In this paper, we propose a framework for the design of genuine globally conservative compact finite difference schemes, which addresses a critical requirement in conservation laws.
Within our framework, we rigorously establish that the discrete conservation law maintains strict conservation for flux functions in polynomial spaces with optimal algebraic order, i.e., the discrete scheme achieves an optimal algebraic precision.
Our work advances the existing conservative compact finite difference schemes, which rely on approaches to maintaining global conservation that are fundamentally consistent with the method proposed by Lele [Lele, J. Comput. Phys., 1992]. 
As an application, we propose an algorithm for designing globally conservative fourth-order schemes, aimed at optimizing resolution and asymptotic stability.
Three schemes are generated using the algorithm, with their excellent performance across multiple aspects validated through numerical experiments.
\end{abstract}

\vspace*{4mm}
\noindent {\bf Keywords:} 
{compact finite difference method,  conservation laws, globally conservative schemes, high resolution, algebraic precision.}

\maketitle

\section{Introduction}
The compact finite difference method has found extensive application across a wide range of fields, including turbulence problems in fluid dynamics \cite{LiMing1995compact, Gamet1999compact, Shah2010upwind, TianZhenfu2011higher}, two-phase flow problems \cite{ZhengFeng2021high}, elliptic interface problems \cite{FengQiwei2021sixth,FengQiwei2022high}, option pricing models \cite{Tangman2008numerical,Roul2021compact}, and fractional equations \cite{CuiMingrong2009compact}. These problems typically exhibit multiscale characteristics and intricate details, demanding high-accuracy, high-resolution, and low-dissipative numerical schemes for effective simulation. By achieving high accuracy with fewer stencil points and offering sufficient flexibility to design stable numerical schemes that balance high resolution and low dissipation, the compact finite difference method emerges as an optimal choice for addressing such challenges.

A suitable compact finite difference scheme typically demands comprehensive consideration of truncation error, resolution, and stability \cite{LELE199216,kim1997implementation}. Truncation error governs the convergence order, with specific constraints enforced through the matching of Taylor expansion coefficients. Resolution error, composed of dispersive and dissipative components, arises from the real part (dispersion) and imaginary part (dissipation) of the pseudo-wavenumber, respectively. Central compact difference schemes inherently feature no dissipative error; however, non-central schemes applied at boundaries introduce both errors. Nevertheless, precise adjustments to the real and imaginary parts of the pseudo-wavenumber enable resolution optimization. The stability of a numerical scheme is evaluated by analyzing the eigenvalues of the full discrete operator: a scheme is deemed stable if its eigenvalues lie within the left half of the complex plane. In practice, stability is closely tied to the specific problem under study. For instance, Brady and Livescu developed high-order stable boundary schemes for explicit and compact finite difference schemes using a numerical stability optimization strategy based on the one-dimensional Euler equations \cite{brady2019high}. Wu et al., meanwhile, optimized their scheme by minimizing the approximation error of a test function’s derivative under given grid parameters, while enforcing stability constraints \cite{wu2024towards}; however, this approach sacrificed resolution optimization, as they posited that system resolution is primarily determined by the interior scheme.

For hyperbolic systems, global conservation is of critical importance. 
However, compact difference schemes do not inherently satisfy this property. 
Qin et al. \cite{qin2021role} investigated fifth-order difference schemes, distinguishing between globally conservative and non-globally conservative formulations that differ solely at boundary grid points. 
The globally conservative scheme demonstrated superior performance over its non-conservative counterpart in terms of numerical error reduction and long-term simulation fidelity. 
Lele \cite{LELE199216} proposed weight coefficients for compact difference schemes near boundaries to enforce compliance with global hyperbolic conservation requirements. 
Nevertheless, these weights lack the properties of genuine integration weights, and their approach only ensures that boundary fluxes are contributed exclusively by boundary points without genuine ensuring discrete conservation in the context of numerical integration. 
Prior studies on the development of conservative compact finite difference schemes, such as \cite{brady2019high}, largely adhere to the same logical framework as proposed by Lele in \cite{LELE199216}.
This work aims to advance their methodology.

In this paper, we propose a general framework for developing genuine globally conservative compact finite difference schemes on uniform grids for conservation laws.
Global conservation of the conservation law is established in terms of the integral formulation. When the continuous conservation law is spatially discretized using a finite difference scheme, the integral corresponding to the continuous global conservation condition is also discretized simultaneously, necessitating a set of integration weights.
To address this issue, we introduce a set of auxiliary variables in conjunction with the integration weights and rigorously enforce discrete global conservation within the framework of numerical integration.
The contributions of this work are summarized as follows:
\begin{enumerate}
	\item[(i)] We demonstrate that when both the numerical schemes applied at the two boundaries and their corresponding integration weights are symmetric, the algebraic order of precision achieves optimality.
	\item[(ii)] Three globally fourth-order schemes are developed within the conservative framework through optimization of resolution capability and asymptotic stability.
\end{enumerate}
Finally, we validate the fourth-order convergence and superior stability of the proposed numerical schemes via numerical experiments.

The remainder of this paper is structured as follows.
In {\bf Section \ref{sec:method}}, we present a framework for developing genuinely globally conservative discrete conservation laws, and within this conservative framework, a fourth-order numerical scheme is introduced.
In {\bf Section \ref{optimization}}, we analyze the resolution and asymptotic stability of the proposed schemes. Additionally, an algorithm is proposed to design the numerical schemes via the optimization of resolution and asymptotic stability.
In {\bf Section \ref{numerical}}, numerical experiments are conducted to validate the order and stability of the proposed numerical schemes.
Finally, concluding remarks are presented in {\bf Section \ref{conclusion}}.

\section{Globally conservative compact finite difference framework}\label{sec:method}
\subsection{Preliminary review}
Consider a one-dimensional computational domain defined as $[0,L]$. The domain is discretized using uniformly spaced grid nodes, denoted by $G=\{x_i:i=0,1,…,N\}$, where $0=x_0<x_1<\cdots<x_N=L$ and $h=x_i-x_{i-1}$ represents the uniform grid spacing.
Compact finite difference schemes, employed to discretize the first-order derivative of an arbitrary function $f(x)$, adhere to a unified formulation\cite{wu2024towards}:
\begin{equation}\label{eq:inner_scheme}
	\sum_{m=-M_l}^{M_r}c_mf'_{i+m}=\frac{1}{h}\sum_{n=-N_l,\neq0}^{N_r}d_n\left(f_{i+n}-f_i\right),
\end{equation}
where $f_i$ and $f_i'$ denote the numerical approximations of $f(x)$ and its derivative $\partial f/\partial x$ at $x_i$, respectively. 
The coefficients $c_m$ and $d_n$ are determined by the prescribed accuracy requirement, while $M_l,M_r,N_l$ and $N_r$ characterize the extent of the grid point stencil (i.e., the number of neighboring nodes involved in the discretization).
For central difference schemes, where $M_l = M_r$ and $N_l = N_r$, the coefficients satisfy symmetry conditions, i.e., $c_m=c_{-m}$ and $d_n=-d_n$. 

By equating the coefficients of the Taylor series expansion, the coefficients of a $k$-th order numerical scheme are required to satisfy the following relationship \cite{wu2024towards}:
\begin{equation}\label{eq:order_relationship}
	\sum_{m=-M_l}^{M_r}m^{j-1}c_m=\frac{1}{j}\sum_{n=-N_l,\neq0}^{N_r}n^jd_n, \quad j=1,2,\cdots,k.
\end{equation}
This condition ensures the scheme achieves the desired order of accuracy by enforcing alignment between the discrete approximation and the Taylor series expansion up to the specified truncation order.

We categorize the grid points into two distinct sets: the interior points, denoted by: $\mg_I = \{x_{l_1},\cdots,x_{N-l_2}\}$, where $l_1$ and $l_2$ are two positive integers, and the boundary points: $\mg_b = \mg - \mg_I$(i.e., the complement of $\mg_I$ in $\mg$).
Furthermore, the boundary set $\mg_b$ is partitioned into two subsets: the left boundary $\mg_b^l$ and the right boundary $\mg_b^r$, such that $\mg_b = \mg_b^l\cup \mg_b^r$(with subscripts $l$ and $r$ denoting "left" and "right", respectively).
On the interior points $\mg_I$, central difference schemes are typically applied due to their inherent low-dissipation characteristics. 
In contrast, while on $\mg_b$, central schemes are invalid on boundary points $\mg_b$ because the stencil asymmetry disrupts the symmetry required for central discretization.
For clarity in subsequent discussions, we reformulate the discretization of $\partial_x f $ into a unified framework:
\begin{equation}\label{Compact_scheme_all}
	\left\{\begin{array}{ll}
		\sum_{j=i-M}^{i+M}a_{ij}f'_j=\frac{1}{h}\sum_{j=i-N}^{i+N}b_{ij}f_{j},  & x_i \in \mg_I, \\
		\sum_{j=0}^{j=S_1^l}a_{ij}f'_{j} = \frac{1}{h}\sum_{j=0}^{j=S_2^l}b_{ij}f_{j}, & x_i\in \mg_b^l, \\
		\sum_{j=0}^{j=S_1^r}a_{N-i,N-j}f'_{N-j} = \frac{1}{h}\sum_{j=0}^{j=S_2^r}b_{N-i,N-j}f_{N-j}, & x_{N-i}\in \mg_b^r.
	\end{array}\right.
\end{equation}
The coefficients in this system (\ref{Compact_scheme_all}) are determined by enforcing the order-accuracy condition (\ref{eq:order_relationship}) and the conservation constraints detailed in the following section.
For simplicity in the following discussion, we denote the scheme defined on $x_i\in \mg_I$ as the interior scheme and the scheme defined on $x_i\in \mg_b$ as the boundary scheme.

\subsection{A globally conservative framework}
For hyperbolic systems, globally conservative approximations are often preferred, as they ensure that numerical solutions preserve key conserved quantities. 
Consider a scalar hyperbolic conservation law of the form:
\begin{equation}\label{convection}
	\frac{\partial u}{\partial t} + \frac{\partial f}{\partial x} = 0, \quad x\in[0,L],
\end{equation}
where $f=f(u)$ denotes the flux function. A defining property of solutions to this equation is that the total variation of $u$ over time is governed exclusively by the flux function $f$ at the domain boundaries. This behavior can be formally demonstrated by integrating Eq. \eqref{convection} over the computational domain:
\begin{equation}\label{integral}
	\frac{d}{dt}\int^{L}_{0}u(x,t)dx = f|_{x=0} - f|_{x=L}.
\end{equation}
The integral formulation explicitly shows that the temporal evolution of the total quantity of $u$ depends only on the flux values at the boundaries.

Given a set of quadrature weights $W=[w_0,w_1,\cdots,w_{N}]^\top$ corresponding to the quadrature points of $\mg$, and define the composite numerical integration rule for any integrable function $f$ on the interval $[0,L]$ as
\begin{equation}\label{numerical_integration}
	I_N[f] \triangleq h\sum_{k=0}^N {w_k} f(x_k). 
\end{equation}
It is well established that a carefully designed configuration of $W$ ensures the numerical integration rule Eq. \eqref{numerical_integration} achieves the desired algebraic accuracy.

Based on Eq. \eqref{numerical_integration}, we discretize the global conservation law Eq. \eqref{integral} into the following discrete form:
\begin{equation}\label{discrete_integral}
	\frac{d}{dt}\sum^{i=N}_{i=0}w_i{u_i(t)}h = {f_0(t) - f_N(t)}, 
\end{equation}
where ${u_i(t)}$ and ${f_i(t)}$ denote the approximations of ${u(x_i, t)}$ and ${f(x_i, t)}$, respectively. {For notational convenience, the time variable $t$ is omitted in the following discussion.}

Applying Eq. \eqref{Compact_scheme_all} to discretize the first-order spatial derivative in Eq. \eqref{convection} yields the following matrix form:
\begin{equation}\label{matrix_form_derivative}
	A F' = \frac{1}{h} B F.
\end{equation}
Here, the entries of $A,B\in \bbR^{(N+1)\times(N+1)}$ are retirved from Eq. \eqref{Compact_scheme_all}, with the remaining unused entries in Eq. \eqref{Compact_scheme_all} set to $0$.
Here, $F = [f_0,f_1,\cdots,f_N]^\top$ represents the column vector storing discrete flux values at grid points $\mg$, and $F' = [f_0',f_1',\cdots,f_N']^\top$ denotes the column vector of nodal derivatives of $f$ at the same grid points.
The $(N+1)\times (N+1)$ coefficient matrices $A$ and $B$ are derived from the compact finite difference schemes given by Eq. \eqref{Compact_scheme_all}.

Substituting Eq. \eqref{matrix_form_derivative} into Eq. \eqref{convection} yields the following semi-discrete form:
\begin{equation}\label{semi_discrete}
	\frac{d}{dt}AU+\frac{1}{h}BF=0,
\end{equation}
where $U=[u_0,u_1,...,u_N]^\top$ denotes the column vector of numerical approximations for $u(x_i)$ at grid points $x_i$.

To enforce the discrete global conservation, we introduce an auxiliary vector $W'=[w_0', w_1', \cdots, w_{N}']^\top$ and employ it as the linear combination coefficient vector to sum the rows of the semi-discrete system Eq. \eqref{semi_discrete}, yielding:
\begin{equation}    
	\frac{d}{dt}\sum^{i=N}_{i=0}w'_i[AU]_{i}h=-\sum^{i=N}_{i=0}w'_i[BF]_i,
\end{equation}
where $[\cdot]_i$ denotes the $i$-th component of the vector.  

To restore the discrete conservation law Eq. \eqref{discrete_integral}, we introduce the following conditions to determine $W'$ based on the weight vector $W$: 
\begin{equation}\label{wp_w_conservationlaw}
	\frac{d}{dt}\sum^{i=N}_{i=0}w_i'[AU]_ih = \frac{d}{dt}\sum^{i=N}_{i=0}w_iu_ih, \quad \sum^{i=N}_{i=0}w_i'[BF]_i=-f_0+f_N.
\end{equation}
These conditions can be compactly represented in matrix form as:
\begin{equation}\label{wp_relation_w}
	W' A=W,\quad W' B = [-1,0,\cdots,0,1].
\end{equation}

\begin{remark}
	The conservation law Eq. \eqref{discrete_integral} represents a constraint imposed on the numerical solutions $\{u_i\}_{i=0,1,\cdots,N}$ rather than the numerical scheme itself. 
	Consequently, irrespective of the specific forms of matrices $A$ and $B$, the relation Eq. \eqref{wp_relation_w} restores the discrete conservation law Eq. \eqref{discrete_integral}, thereby ensuring global conservation inherently.
\end{remark}

Naturally, determining the space of functions on which the numerical integration scheme Eq. \eqref{numerical_integration} works accurately, specifically, its algebraic precision, is of critical importance.
As the primary contribution of this work, we demonstrate that the order of algebraic precision of the numerical integration scheme Eq. \eqref{numerical_integration} attains its optimal value. 
This optimality arises from the symmetry of the underlying numerical schemes proposed for $x_l\in \mg_b^l$ and $x_{N-l}\in\mg_b^r$.
The corrsponding result is formally stated in {\bf Theorem \ref{theorem:algebraic_order_general}}. 

\begin{theorem}\label{theorem:algebraic_order_general}
	Denote $T_I^i[f]$, $T_L^i[f]$ and $T_R^i[f]$ as the truncation errors of Eq. \eqref{matrix_form_derivative} on $x_i\in$ $\mg_I$, $\mg_b^l$ and $\mg_b^r$, respectively.
	Let (C1)-(C3) be fulfilled.
	\begin{enumerate}
		\item[(C1)]$l_1 = l_2 = l$, $S_1^l = S_1^r = S_1$, $S_2^l=S_2^r=S_2$. 
		\item[(C2)]$a_{ij} = a_{N-i,N-j}$, $b_{ik} = - b_{N-i,N-k}$, $w_i = w_{N-i}$, for $0 \leq i \leq l-1$, $0 \leq j \leq S_1$,  $0 \leq k \leq S_2$.
		\item[(C3)]$\exists k \in \bbZ$, $T_I^i[f] \sim O(h^{2k})$, $i=l,\cdots, N-l$, and $T^i_L[f] \sim T^i_R[f] \sim O(h^{K})$, $\forall K\geq 2k-1$, $i=0,1,\cdots,l-1$.
	\end{enumerate}
	Then, Eq. \eqref{numerical_integration} with the integration weight vector $W$ satisfy Eq. \eqref{wp_relation_w} posesses $2k-1$-th order algebraic precision.   
\end{theorem}
\begin{proof}
	Clearly, Eq. \eqref{wp_relation_w} can be rewritten as $WA^{-1}B = [-1,0,\cdots,0,1]$.
	Substituting the above relation in the second equation of Eq. \eqref{wp_w_conservationlaw}, using (C1) and replacing $f_i$ with $f(x_i)$, one notices 
	\[
	hW{\bbf}' = hW \left(\frac{1}{h}A^{-1} B\bbf\right) + R_N  = -f(x_0) + f(x_N) + R_N, 
	\] 
	where \(\bbf' = [f'(x_0), f'(x_1), \dots, f'(x_N)]^\top\) denotes the column vector of the first derivatives of the function \(f\) evaluated at points \(x_0, x_1, \dots, x_N\), \(\bbf = [f(x_0), f(x_1), \dots, f(x_N)]^\top\) denotes the column vector of function values at these points, and 
	\begin{equation}\label{R_N}
		R_N = h \sum_{i=l}^{N-l} w_i T_I^i + h \sum_{i=0}^{l-1} w_i \left( T_L^i + T_R^i \right).
	\end{equation}
	According to (C2)-(C3), and recalling Eq. \eqref{Compact_scheme_all}, one verifies 
	\[
	T_I^i =f^{(2k+1)}(\xi_i) O\left(h^{2k}\right), \quad \xi_i \in (x_{i-1},x_{i+1}),
	\]
	and 
	\begin{equation}\label{T_L}\begin{split}
			T_L^i &= \sum_{k=0}^{S_1} a_{ik} f'(x_k) - \frac{1}{h}\sum_{k=0}^{S_2} b_{ik} f(x_k) \\
			&= h^{K}\left(\sum_{k=0}^{S_1}\frac{(k-i)^{K}}{K!} a_{ik}f^{(K+1)}(\xi_i^k) - \sum_{k=0}^{S_2}\frac{(k-i)^{(K+1)}}{(K+1)!} b_{ik}f^{(K+1)}(\eta_i^k) \right) , \\
			& \quad \quad  \quad \quad \quad  \quad \quad \quad  \quad\text{with}\quad \xi_i^k\in (x_0,x_{S_1}),\quad \eta_i^k \in (x_0,x_{S_2}),
	\end{split}\end{equation}
	\begin{equation}\label{T_R}\begin{split}
			T_R^i &= \sum_{k=0}^{S_1} a_{ik} f'(x_{N-k}) - \frac{1}{h}\sum_{k=0}^{S_2}(- b_{ik}) f(x_{N-k}) \\
			&= (- h)^{K}\left(\sum_{k=0}^{S_1}\frac{(k-i)^{K}}{K!} a_{ik}f^{(K+1)}(\xi_{N-i}^k) - \sum_{k=0}^{S_2}\frac{(k-i)^{(K+1)}}{(K+1)!} b_{ik}f^{(K+1)}(\eta_{N-i}^k) \right), \\
			& \quad \quad \quad  \quad \quad  \quad \quad \quad  \quad \text{with} \quad \xi_i^k\in (x_{N-S_1},x_N),\quad\eta_i^k \in (x_{N-S_2},x_N).
	\end{split}\end{equation}
	For $K>2k-1$, $f^{(2k+1)}(x)$ and $f^{(K+1)}(x)$ both vanish if $f\in P_{2k}(x)$, which implies that the numerical integration formula Eq. \eqref{numerical_integration} achieves $(2k-1)$-th order algebraic precision.
	When $K = 2k-1$, substitute $f = \left(x-\frac{L}{2}\right)^{2k}$ into $T_L^i$ and $T_R^i$. 
	Due to symmtry of the discretization schemes applied to $x_i$ and $x_{N-i}$, $i=0,1,\cdots,l-1$, $\xi_i$ in Eq. \eqref{T_L} and $\xi_{N-i}$ in Eq. \eqref{T_R} are symmetric about $x=\frac{L}{2}$. 
	Similarly, $\eta_i$ in Eq. \eqref{T_L} and $\eta_{N-i}$ in Eq. \eqref{T_R} are symmetric about $x=\frac{L}{2}$, too.
	As a result, $T_l^i[f]$ and $T_R^i[f]$ cancle each other out in Eq. \eqref{R_N}. 
	Thus, it follows that $T_N[f]=0$ for all $f\in P_{2k}$, which implies that the numerical integration scheme Eq. \eqref{numerical_integration} attains $(2k-1)$-th order algebraic precision.
\end{proof}

\begin{remark}
	Regardless of (C1) and (C2), (C3) demonstrates that the lower bound of the algebraic precision order is dictated by the minimum of the consistency orders of the finite difference schemes applied to all grid nodes in $\mg$.
	(C1) and (C2) stipulate a symmetry requirement for the schemes applied to $\mg_b^l$ and $\mg_b^r$, which is critical for maximizing the order of algebraic precision.
\end{remark}

\subsection{Globally fourth-order conservative schemes}\label{sec:Construction}
Henceforth, (C1) and (C2) stated in Theorem \ref{theorem:algebraic_order_general} are assumed to hold naturally.
To gain fourth-order accuracy globally, we adopt the fourth-order tridiagonal compact finite difference scheme for approximating $f'(x_i)$, for $x_i \in \mg_I$:
\begin{equation}\label{eq:4_order_inner_compact_scheme}
	\frac{1}{6}f'_{i-1}+\frac{2}{3}f'_i+\frac{1}{6}f'_{i+1}=\frac{1}{2h}(f_{i+1}-f_{i-1}).
\end{equation}
As pointed out in \cite{gustafsson1975convergence}, the fully disretizated system, involving Eq. \eqref{eq:4_order_inner_compact_scheme} on $\mg_I$, can achieve globally fourth-order accuracy by applying only third-order schemes on $\mg_b$, under certain regularity assumptions. 
Therefore, we let $S=3$.
The coefficients $a_{ij}$ and $b_{ij}$ for the boundary schemes need to be optimized in accordance with the order and stability conditions, the high-resolution condition, and most critically, the global conservation law.
These issues will be addressed in the following sections.

Recall Eq. \eqref{eq:4_order_inner_compact_scheme}: $l\geq 1$ must hold, as otherwise nonexistent grid points outside the boundaries would be required.
Furthermore, $l<4$ is necessary, since violating the second condition of Eq. \eqref{wp_relation_w} would occur otherwise.
This is because when $l\geq 4$, the sum of the entries in at least one column of the matrix B does not vanish, rendering $W'$ unsolvable.

In this paper, we investigate three schemes denoted as P1, P2, and P3, where $l$ is set to $1$, $2$, and $3$, respectively. 
In what follows, we take $l=3$ as an example to illustrate the method for constructing the boundary schemes.

Specifically, when $l=3$ 
\begin{equation}\label{matrix_A}
	A=\begin{bmatrix}
		a_{00}&a_{01}&a_{02}&a_{03}&\cdots&0&0&0&0\\
		a_{10}&a_{11}&a_{12}&a_{13}&0&\cdots&0&0&0\\
		a_{20}&a_{21}&a_{22}&a_{23}&0&0&\cdots&0&0\\
		0&0&\frac{1}{6}&\frac{2}{3}&\frac{1}{6}&0&0&\cdots&0\\
		\vdots&\ddots&\ddots&\ddots&\ddots&\ddots&\ddots&\ddots&\vdots\\
		0&\cdots&0&0&\frac{1}{6}&\frac{2}{3}&\frac{1}{6}&0&0\\
		0&0&\cdots&0&0&a_{23}&a_{22}&a_{21}&a_{20}\\
		0&0&0&\cdots&0&a_{13}&a_{12}&a_{11}&a_{10}\\
		0&0&0&0&\cdots&a_{03}&a_{02}&a_{01}&a_{00}
	\end{bmatrix},
\end{equation}
and
\begin{equation}\label{matrix_B}
	B=\begin{bmatrix}
		b_{00}&b_{01}&b_{02}&b_{03}&\cdots&0&0&0&0\\
		b_{10}&b_{11}&b_{12}&b_{13}&0&\cdots&0&0&0\\
		b_{20}&b_{21}&b_{22}&b_{23}&0&0&\cdots&0&0\\
		0&0&-\frac{1}{2}&0&\frac{1}{2}&0&0&\cdots&0\\
		\vdots&\ddots&\ddots&\ddots&\ddots&\ddots&\ddots&\ddots&\vdots\\
		0&\cdots&0&0&-\frac{1}{2}&0&\frac{1}{2}&0&0\\
		0&0&\cdots&0&0&-b_{23}&-b_{22}&-b_{21}&-b_{20}\\
		0&0&0&\cdots&0&-b_{13}&-b_{12}&-b_{11}&-b_{10}\\
		0&0&0&0&\cdots&-b_{03}&-b_{02}&-b_{01}&-b_{00}
	\end{bmatrix}.
\end{equation}
Matrices for P1 and P2 can be found in the Appendix section.
By matching the Taylor expansion series of boundary schemes up to the order of $O(h^2)$, one sees that 
\begin{equation}\label{taylor:i=0}
	x_0 : \quad 
	\left\{
	\begin{aligned}
		0 & = b_{00} +b_{01} + b_{02} +b_{03},\\
		a_{00}+a_{01}+a_{02}+a_{03} & = b_{01}+2b_{02}+3b_{03},\\
		a_{01}+2a_{02}+3a_{03} & = \frac{1}{2}b_{01}+2b_{02}+\frac{9}{2}b_{03},\\
		\frac{1}{2}a_{01}+2a_{02}+\frac{9}{2}a_{03} & = \frac{1}{6}b_{01}+\frac{4}{3}b_{02}+\frac{9}{2}b_{03},
	\end{aligned}			
	\right.
\end{equation}
\begin{equation}\label{taylor:i=1}
	x_1 : \quad 
	\left\{
	\begin{aligned}
		0 & = b_{10} +b_{11} + b_{12} +b_{13},\\
		a_{10}+a_{11}+a_{12}+a_{13} & = -b_{10}+b_{12}+2b_{13},\\
		-a_{10}+a_{12}+2a_{13} & = \frac{1}{2}b_{10}+\frac{1}{2}b_{12}+2b_{13},\\
		\frac{1}{2}a_{10}+ \frac{1}{2}a_{12}+2a_{13} & = -\frac{1}{6}b_{10}+\frac{1}{6}b_{12}+\frac{4}{3}b_{13},
	\end{aligned}			
	\right.
\end{equation}
\begin{equation}\label{taylor:i=2}
	x_2 : \quad 
	\left\{
	\begin{aligned}
		0 & = b_{20} +b_{21} + b_{22} +b_{23},\\
		a_{20}+a_{21}+a_{22}+a_{23} & = -2b_{20}-b_{21}+b_{23},\\
		-2a_{20}-a_{21}+a_{23} & = 2b_{20}+\frac{1}{2}b_{21}+\frac{1}{2}b_{23},\\
		2a_{20}+\frac{1}{2}a_{21}+\frac{1}{2}a_{23} & = -\frac{4}{3}b_{20}-\frac{1}{6}b_{21}+\frac{1}{6}b_{23}.
	\end{aligned}			
	\right.
\end{equation}
In addition, according to Eq. \eqref{wp_relation_w}, 
\begin{equation}\label{bsh1}
	W' A=W  \quad \Longrightarrow \quad
	\left\{
	\begin{aligned}
		w'_0a_{00}+w'_1a_{10}+w'_2a_{20}&=w_0,\\
		w'_0a_{01}+w'_1a_{11}+w'_2a_{21}&=w_1,\\
		w'_0a_{02}+w'_1a_{12}+w'_2a_{22}&=w_2-\frac{1}{6},\\
		w'_0a_{03}+w'_1a_{13}+w'_2a_{23}&=w_3-\frac{5}{6},
	\end{aligned}
	\right.
\end{equation}
and
\begin{equation}\label{bsh2}
	W' B = [-1,0,\cdots,0,1] \quad  \Longrightarrow \quad
	\left\{
	\begin{aligned}
		w'_0b_{00}+w'_1b_{10}+w'_2b_{20}&=-1,\\
		w'_0b_{01}+w'_1b_{11}+w'_2b_{21}&= 0,\\
		w'_0b_{02}+w'_1b_{12}+w'_2b_{22}&=\frac{1}{2},\\
		w'_0b_{03}+w'_1b_{13}+w'_2b_{23}&=\frac{1}{2},
	\end{aligned}
	\right.
\end{equation}
must be satisfied simultaneously alongside Eq. \eqref{taylor:i=0}-Eq. \eqref{taylor:i=2} to ensure global conservation.

The other weights can be set as $w_i' = 1,\, i=3,\cdots, N-3$ and $w_i = 1,\, i=4,\cdots, N-4$. 
The auxiliary setting for $w$ is natural because any numerical integration formula should be at least accurate for constant integrands.
Then $w'$ is set accordingly, as the sum of entries of $i$th column of $A$ equals $1$, $4\leq i\leq N - 4$.

\begin{remark}
	The necessary conditions for the solvability of P1, P2, and P3 are proven to be identical.
	The detailed calculation and proof process is presented in Section \ref{app:necessary} of the Appendix.
\end{remark}

			\section{Analysis and optimization}\label{optimization}
			\subsection{Fourier analysis}
			In this section, we analyze the resolution, i.e., the dissipation and dispersion, of the schemes applied to $\mg_b$.
			To facilitate analysis, we use finite wave numbers and define the Fourier expansion of a function $f(x) 
			\in L^1[0,L]$ as 
			\begin{equation}\label{fourier}
				f(x)
				=\sum_{k=-N/2}^{N/2-1}\hat{f}_k\exp\left(\mathrm{i}\omega_k x\right),
			\end{equation}
			where $\mathrm{i}=\sqrt{-1}$, $k$ denotes the wavenumber, $\omega_k=2\pi k/L$ is the frequency.
			Taking the first-order derivative of Eq. \eqref{fourier}, we obtain the exact expansion coefficients of the derivative function.
			To quantify the numerical error in the frequency domain, we substitute Eq. \eqref{fourier} into the boundary discretization schemes, leading to
			\begin{equation}\label{bar-omega}
				\mathrm{i}\bar{\omega}_k= \frac{A(\omega_k)+\mathrm{i}B(\omega_k)}{C(\omega_k)+\mathrm{i}D(\omega_k)},
			\end{equation}
			where $\bar{\omega}_k$ denotes the numerical frequency. 
			Owing to the decoupling of different frequencies in Eq. \eqref{bar-omega}, we omit the subscript $k$ in the following analysis.
			Specifically, $A(\omega)$,$B(\omega)$,$C(\omega)$,$D(\omega)$ are defined according to the underlying scheme.
			For example, 
			\begin{equation}
				x_0:\begin{cases}
					A(\omega)&= b_{00}+b_{01}\cos(\omega)+b_{02}\cos(2\omega)+b_{03}\cos(3\omega),\\
					B(\omega)&= b_{01}\sin(\omega)+b_{02}\sin(2\omega)+b_{03}\sin(3\omega),\\
					C(\omega)&= a_{00}+a_{01}\cos(\omega)+a_{02}\cos(2\omega)+a_{03}\cos(3\omega),\\
					D(\omega)&= a_{01}\sin(\omega)+a_{02}\sin(2\omega)+a_{03}\sin(3\omega),
				\end{cases}
			\end{equation}
			\begin{equation}
				x_1:\begin{cases}
					A(\omega) &= b_{10}\cos(\omega)+b_{11}+b_{12}\cos(\omega)+b_{13}\cos(2\omega),\\
					B(\omega) &= -b_{10}\sin(\omega)+b_{12}\sin(\omega)+b_{13}\sin(2\omega),\\
					C(\omega) &= a_{10}\cos(\omega)+a_{11}+a_{12}\cos(\omega)+a_{13}\cos(2\omega),\\
					D(\omega) &= -a_{10}\sin(\omega)+a_{12}\sin(\omega)+a_{13}\sin(2\omega),
				\end{cases}
			\end{equation}
			\begin{equation}
				x_2:\begin{cases}
					A(\omega) &= b_{20}\cos(2\omega)+b_{21}\cos(\omega)+b_{22}+b_{23}\cos(\omega),\\
					B(\omega) &= -b_{20}\sin(2\omega)-b_{21}\sin(\omega)+b_{23}\sin(\omega),\\
					C(\omega) &= a_{20}\cos(2\omega)+a_{21}\cos(\omega)+a_{22}+a_{23}\cos(\omega),\\
					D(\omega) &= -a_{20}\sin(2\omega)-a_{21}\sin(\omega)+a_{23}\sin(\omega).
				\end{cases}
			\end{equation}

			\subsection{Optimization}\label{sec:Optimization}
			In this section, we introduce an optimization model for determining coefficients in the construction of boundary schemes.
			The optimization model will account for both maximizing resolution and ensuring asymptotic stability.
			
			Let \(\text{Re}(a)\) and \(\text{Im}(a)\) denote the real and imaginary parts of the complex number $a$, respectively.
			Define 
			\begin{equation}
				\varepsilon_{R}(\omega) = \left|\frac{\text{Re}[\bar{\omega}(\omega)] - \omega}{\omega}\right|, \quad \varepsilon_{I}(\omega) = \left|\frac{\text{Im}[\bar{\omega}(\omega)]}{\omega}\right|, 
			\end{equation}
			where $\varepsilon_{R}(\omega)$ denotes the dispersive error while $\varepsilon_{I}(\omega) $ denotes the dissipative error. 
			We define the critical frequency as
			\begin{equation}
				\omega^\sigma = \frac{1}{2}(\omega_{R}^\sigma + \omega_{I}^\sigma),
			\end{equation}
			where
			\begin{equation}
				\omega_{R}^\sigma = \min\{\omega \mid \varepsilon_{R}(\omega) = \sigma, 0 < \omega < \pi\}, \quad \omega_{I}^\sigma = \min\{\omega \mid \varepsilon_{I}(\omega) = \sigma, 0 < \omega < \pi\}.
			\end{equation}
			It is not hard to verify that \(\omega^\sigma\) is the critical point below which the error is consistently low(less than $\sigma$).
			Noticing that both $\varepsilon_{R}(\omega)$ and $\varepsilon_{I}(\omega)$ might grow unbounded when the frequency increases beyond the value.
			
			Since the coefficients in the boundary schemes are interrelated, it is imperative to collectively consider the "average" effect of optimizing the schemes applied at $\mg_b$, rather than optimizing each one in isolation.
			For each $x_i \in \mg_b^l$, we denote $\omega_{i}^{\sigma_i}$ as the correponding critical wave frequency, and define the average resolution as
			\begin{equation}\label{omega_f}
				\omega_f = \frac{1}{3}(\omega_{0}^{\sigma_0}+\omega_{1}^{\sigma_1}+\omega_{2}^{\sigma_2}).
			\end{equation}
			where the value of \(\sigma_i\) is retrieved from \cite{kim2007optimised} with \(\sigma_0 = 0.003\), \(\sigma_1 = 0.002\), and \(\sigma_2 = 0.001\).
			
			%
				To define the optimization problem, we consider the aymptotic stability of Eq. \eqref{semi_discrete}.  
				Consider the equivalent IBVP,
				\begin{equation}\label{IBVP}\begin{split}
						& \frac{\partial u}{\partial t} + \frac{\partial u}{\partial x} = 0, \quad x \in [0, 1], \quad t \geq 0,  \\
						&u(x,0) = u_0(x), \quad u(0,t)=0.
				\end{split}\end{equation}
				Applying Eq. \eqref{matrix_form_derivative} to discretize Eq. \eqref{IBVP}, and using the boundary condition, one obtains 
				\begin{equation}\label{asymptotic_system}
					\frac{d\hat{U}}{dt} = Q\hat{U},
				\end{equation}
				where $Q\in \bbR^{N\times N}$ is the submatrix obtained by eliminating the first row and column of the original coefficient matrix. 
			
			According to \cite{strang2000linear}, the system Eq. \eqref{asymptotic_system} is asymptotically stable if and only if
			\begin{equation}
				\text{Re}(\lambda) \leq 0, \quad \forall \lambda\in \Lambda[Q],
			\end{equation}
			where $\Lambda[Q]$ denotes the set of eigenvalues of $Q$.
			In order to obtain the optimal boundary schemes, we define the optimization problem
			\begin{equation}\label{youhua}
				\begin{aligned}
					\text{maximize} \quad & \omega_f, \\
					\text{subject to} \quad & \max_{\lambda} \text{Re}(\lambda) \leq 0, \\
					& w_i > 0, \quad i = 1, 2, 3, \\
					& w_0 \in (0, 10), \quad a_{03}, b_{03}, a_{13}, b_{13}, w'_0 \in (-10, 10).
				\end{aligned}
			\end{equation}
			{For detailed parameter dependency relationships, see formulas Eq. \eqref{P3_w1p_w2p} and Eq. \eqref{P3_relationship} in Section \ref{app:dependency} of the Appendix.
				The algorithm for the construction of $\omega_f$ for P1, P2 and P3 is designed and stated in Section \ref{app:algorithm} of the Appendix.
				To conduct the optimization, we utilize the global optimization algorithm(scipy.optimize.differential\_evolution) developed in  \cite{storn1997differential} to solve Eq. \eqref{youhua}, for its ability of handling nonlinear and non-convex optimization problems efficiently. }
			
			We conduct experiments to determine $A_i$, $B_i$, $W$ and $W'$ for P1, P2 and P3.
			The number of grid points for optimization is set to $100$.
			The corresponding results are systematically presented in Tables \ref{tab:P1}--\ref{tab:P3}.
			In Figure \ref{fig:eigenvalues}, we also show the real and imaginary parts of each eigenvalue in the set $\Lambda[Q]$, for $N = 50, 100, 200$. 
			Obviously, the real part of all eigenvalues is strictly negative, thereby confirming the numerical stability of the schemes.
			\begin{table*}[!t]
				\caption{\label{tab:P1}Coefficients and weights for P1.}
				\resizebox{\textwidth}{!}{
					
					\centering
					\begin{tabular}{cccc}
						\hline
						$w_0=0.365512831337005295040$ & $w_1=1.19512817265565063352$ & $w_2=0.92987182734434925546$ & $w_3=1.00948716866299470496$\\
						$b_{00}=-2.51450663294882081900$ & $b_{01}=2.51450663294882081900$ & $b_{02}=0$ & $b_{03}=0$\\
						$a_{00}=1$ & $a_{01}=1.819471046485240606224$ & $a_{02}=-0.35267551059842805472$ &$a_{03}=0.04771109706200871159$ \\
						$w_0'=0.19884616467033863763$ & - & - & - \\
						\hline
					\end{tabular}
				}
			\end{table*}
			\begin{table*}[!t]
				\caption{\label{tab:P2}Coefficients and weights for P2.}
				\resizebox{\textwidth}{!}{
					\centering
					\begin{tabular}{cccc}
						\hline
						$w_0=0.35520684553103798464$ & $w_1=1.22604613007355256471$ & $w_2=0.89895386992644732427$ & $w_3=1.01979315446896201536$\\
						$b_{00}=0$ & $b_{01}=26.03939124025922779992$ & $b_{02}=-16.20452306655850804873$ & $b_{03}=-9.83486817370072152755$\\
						$a_{00}=1$ & $a_{01}=-13.03017400229961886282$ & $a_{02}=-20.91617294263996740256$ &$a_{03}=-2.92791246902036483846$ \\
						$b_{10}=-0.76360320980590068451$ & $b_{11}=0$ & $b_{12}=0.61939982125193304707$ & $b_{13}=0.14420338855396769295$\\
						$a_{10}=0.25657462461142366283$ & $a_{11}=1$ & $a_{12}=0.35679065966631573481$ &$a_{13}=0.05804452388802976148$ \\
						$w_0'=0.01920167777297939610$ & $w_1'= 1.30958066592489652535$ & - & - \\
						\hline
					\end{tabular}
				}
			\end{table*}
			\begin{table*}[!t]
				\caption{\label{tab:P3}Coefficients and weights for P3.}
				\resizebox{\textwidth}{!}{
					\centering
					\begin{tabular}{cccc}
						\hline
						$w_0=0.26663842939298731949$ & $w_1=1.49175137848770456017$ & $w_2=0.63324862151229532881$ & $w_3=1.10836157060701268051$\\
						$b_{00}=0$ & $b_{01}=27.88752780480513493444$ & $b_{02}=-19.15566094564050914073$ & $b_{03}=-8.73186685916462579371$\\
						$a_{00}=1$ & $a_{01}=-13.89214768040508829472$ & $a_{02}=
						-21.41292977597984048543$ &$a_{03}=-2.31431720758483283618$ \\
						$b_{10}=-2.41737032042215105321$ & $b_{11}=0$ & $b_{12}=-6.50993055450477520196$ & $b_{13}=8.92730087492692625517$\\
						$a_{10}=1.29886300269147980657$ & $a_{11}=1$ & $a_{12}=8.23654271107762525617$ &$a_{13}=3.22663580200212241067$ \\
						$b_{20}=-0.16112692262471739468$ & $b_{21}=-0.60739819861958466163$ & $b_{22}=0$ & $b_{23}=0.76852512124430449880$\\
						$a_{20}=0.06337857839129412696$ & $a_{21}=0.37486590693825938558$ & $a_{22}=1$ &$a_{23}=0.25993267978377243566$ \\
						$w_0'=4.16532467660117156072$ & $w_1'= -12.33339535057290881070$ & $w_2' = 191.24292432666243257700$ & - \\
						\hline
					\end{tabular}
				}
			\end{table*}

			Finally, we demonstrate the resolution of P1, P2, and P3 in Figures \ref{fig:resolution-real}--\ref{fig:resolution-imag}.
			These figures further illustrate the resolution characteristics by comparing the pseudo-wavenumber components \(\text{Re}[\bar{\omega}(\omega)]\) and \(\text{Im}[\bar{\omega}(\omega)]\) for schemes \(P1\), \(P2\), and \(P3\) at the three nodes neighboring the boundary. 
			For benchmark comparison, we include scheme \(T4\) from \cite{brady2019high}, which employs boundary schemes at nodes \(i=0,1\) and  the interior scheme at node \(i=2\); this configuration inherently produces a zero imaginary component (\(\text{Im}[\bar{\omega}] = 0\)) in its resolution characteristics. 
			The total resolution \(\omega_f\) for the optimized schemes \(P1\), \(P2\), and \(P3\) is approximately \(0.9268\), \(0.9425\), and \(0.9737\), respectively. 
			In comparison, the total resolution in \cite{brady2019high} is approximately \(0.3927\). 
			This demonstrates the improved performance of the optimized boundary schemes.
			

			\begin{figure}[!t]
				\centering
				\begin{subfigure}{.32\textwidth}
					\centering
					\includegraphics[width=\textwidth]{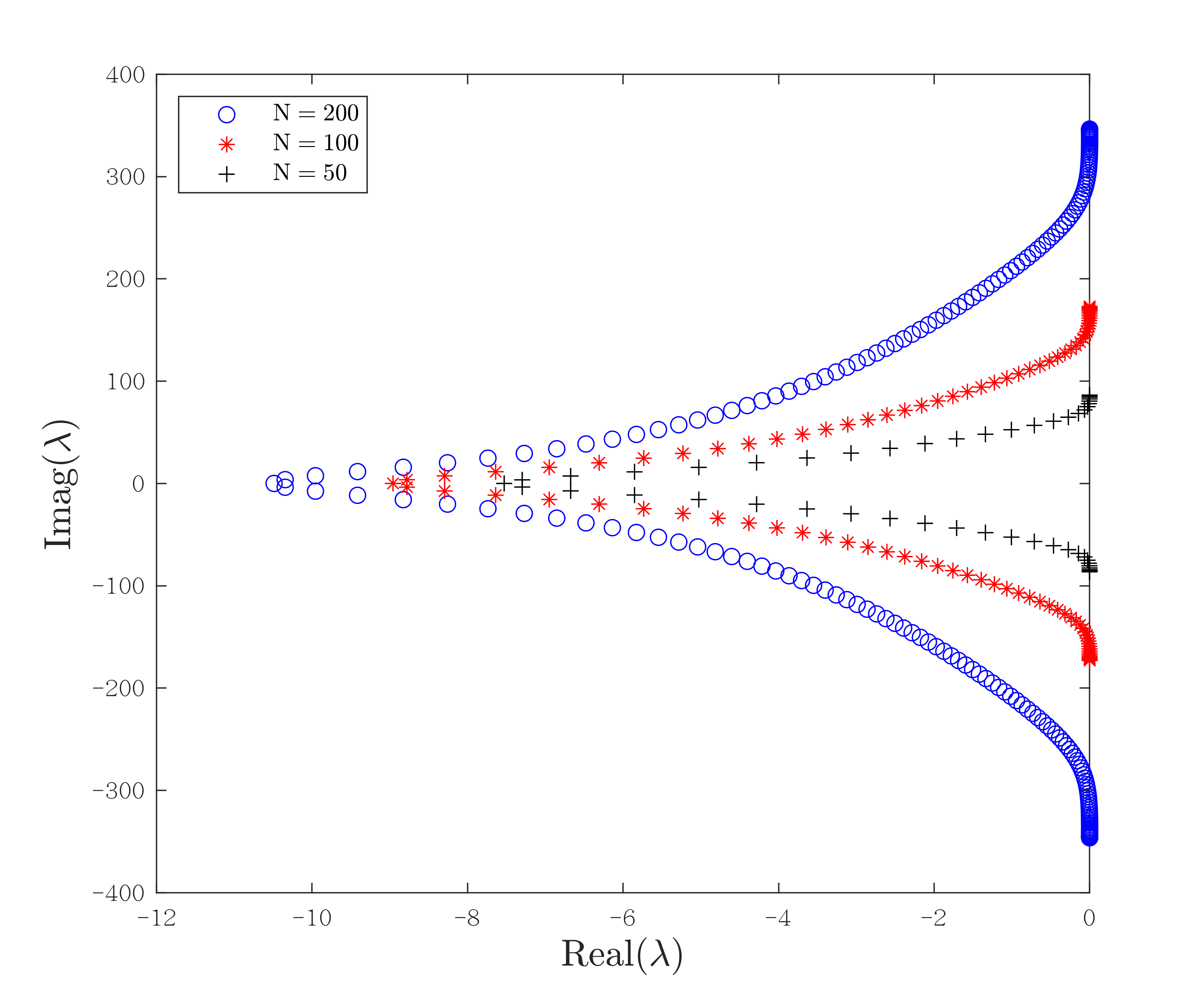}
					\caption{$P1$} 
				\end{subfigure}%
				\begin{subfigure}{.32\textwidth}
					\centering
					\includegraphics[width=\linewidth]{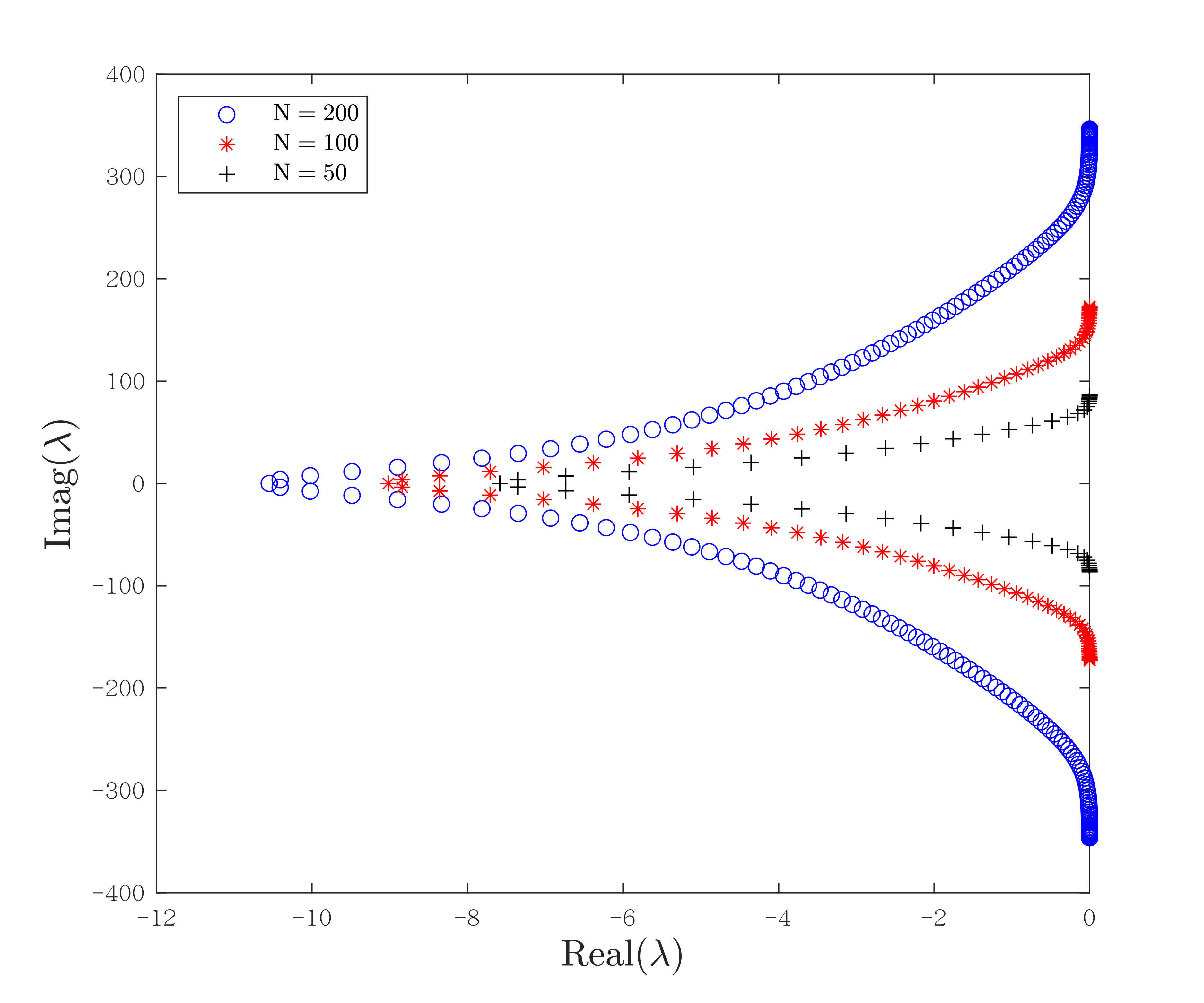}
					\caption{$P2$} 
				\end{subfigure}%
				\begin{subfigure}{.32\textwidth}
					\centering
					\includegraphics[width=\linewidth]{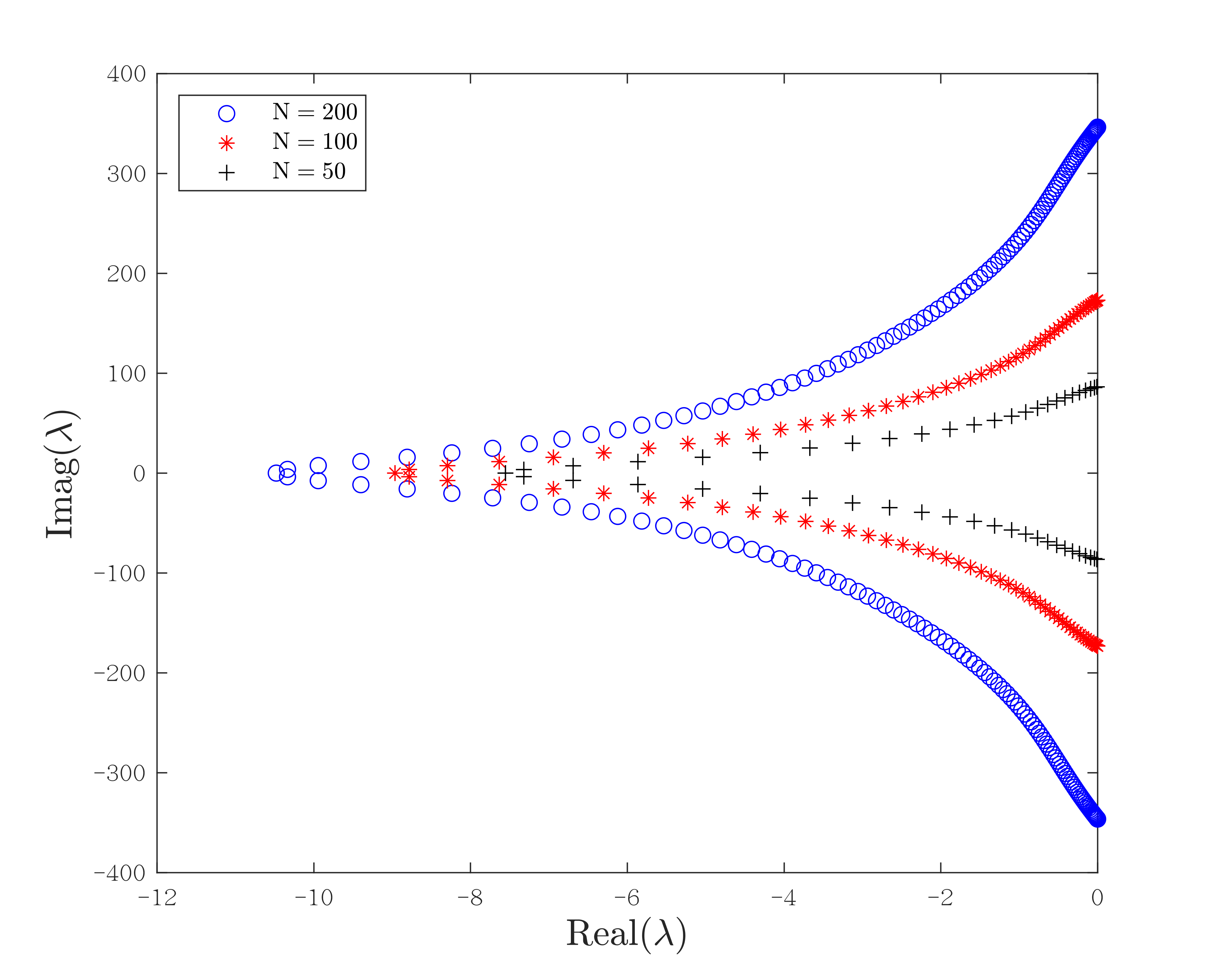}
					\caption{$P3$} 
				\end{subfigure}
				\caption{Distribution of eigenvalues of $Q_{P1}$, $Q_{P2}$ and $Q_{P3}$.}
				\label{fig:eigenvalues}
			\end{figure}
			
			\begin{figure}[!t]
				\begin{subfigure}{0.32\textwidth}
					\centering
					\includegraphics[width=\linewidth]{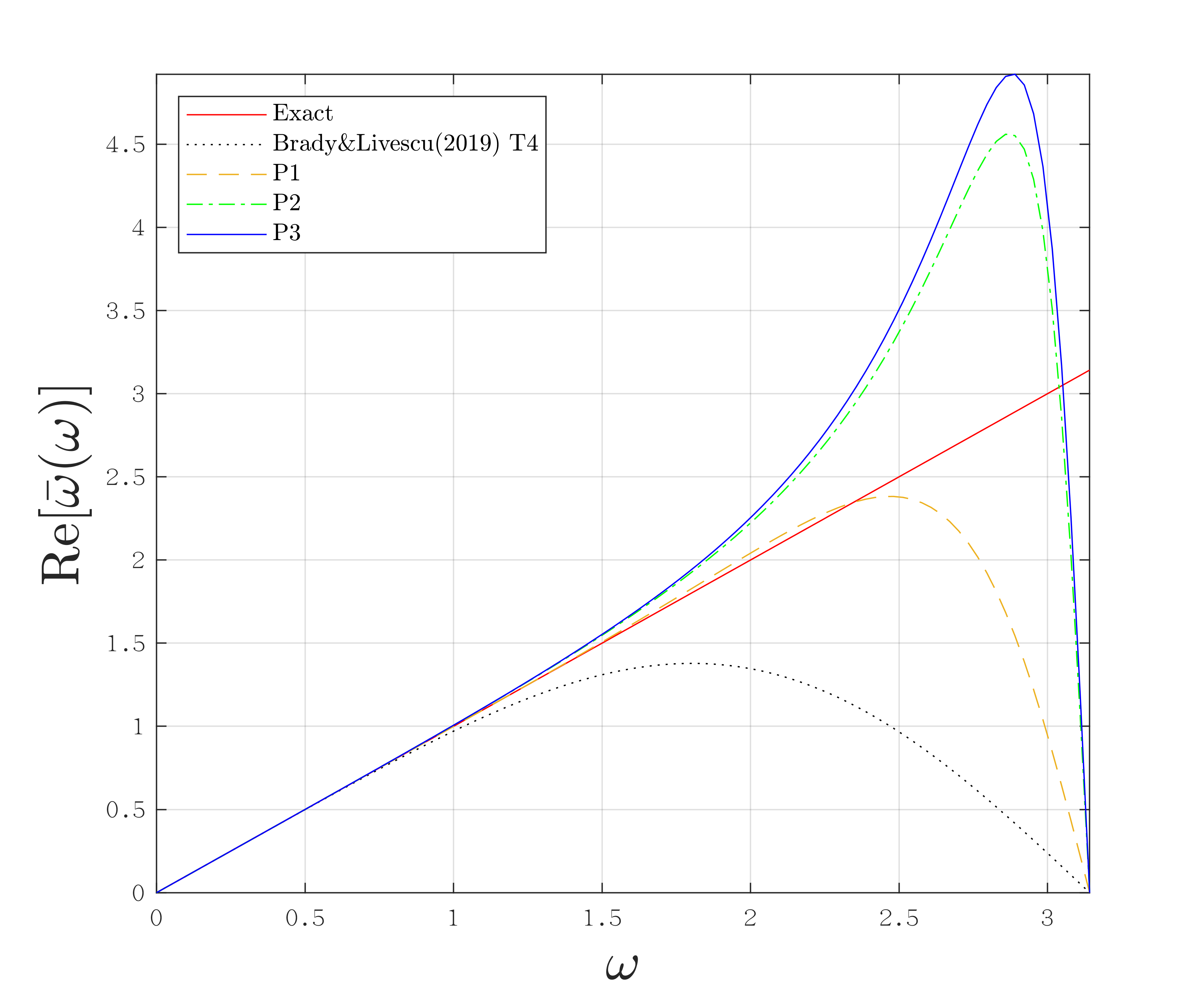}
					\caption{ $i=0$}
					\label{fig:resolution-real-i0}
				\end{subfigure}
				\begin{subfigure}{0.32\textwidth}
					\centering
					\includegraphics[width=\linewidth]{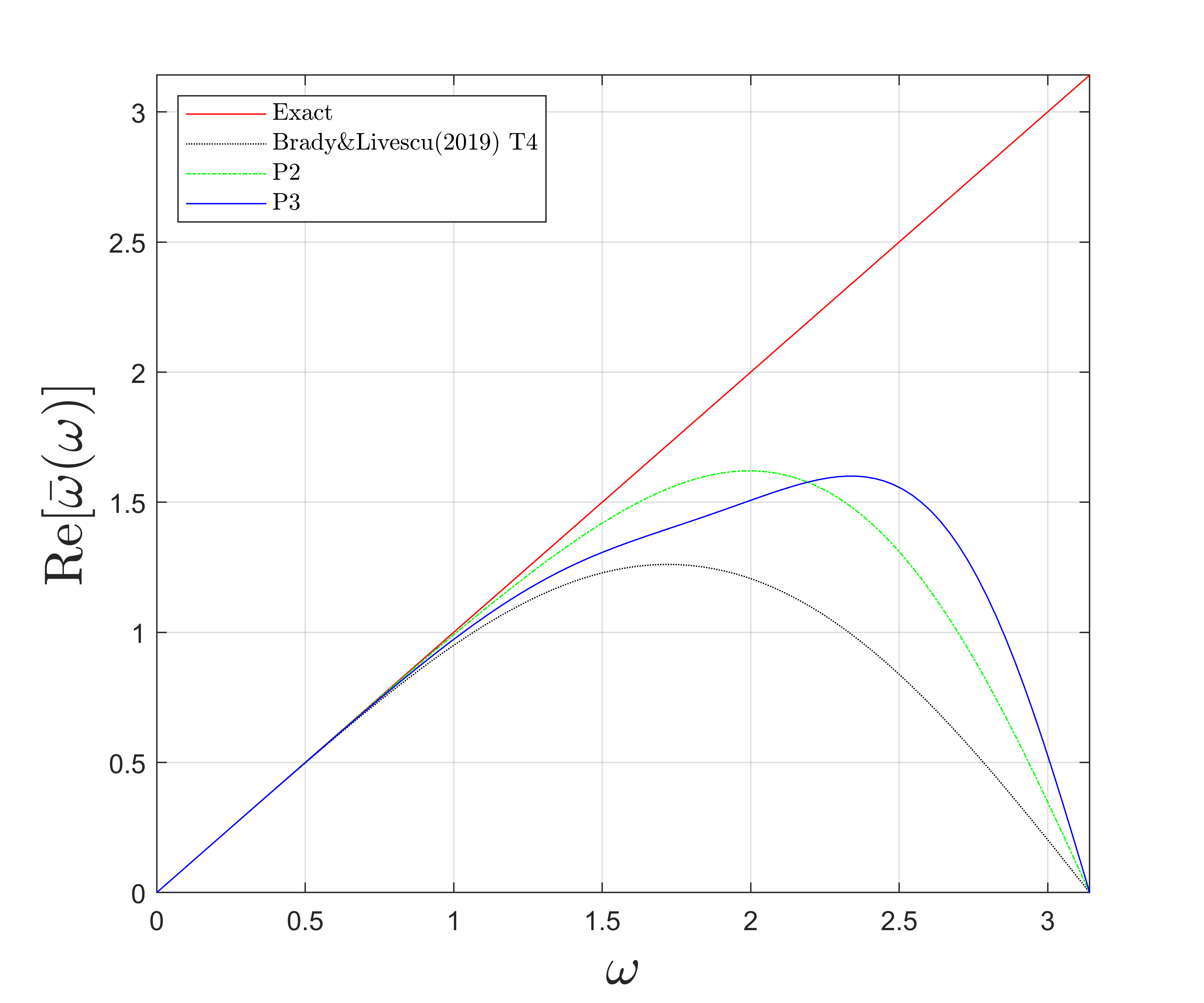}
					\caption{$i=1$}
					\label{fig:resolution-real-i1}
				\end{subfigure}
				\begin{subfigure}{0.32\textwidth}
					\centering
					\includegraphics[width=\linewidth]{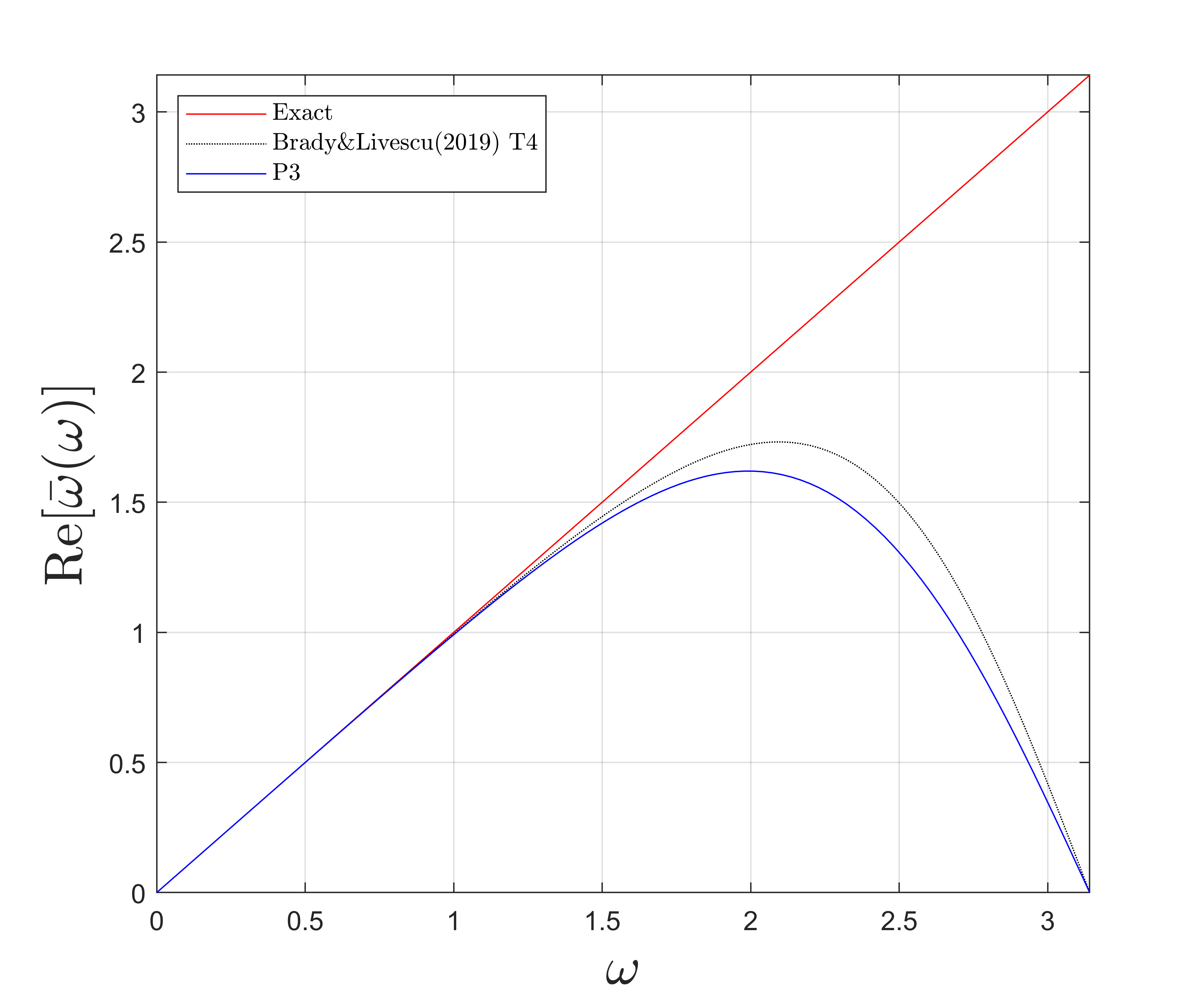}
					\caption{ $i=2$}
					\label{fig:resolution-real-i2}
				\end{subfigure}
				\caption{The real parts of pseudo-wavenumber at nodes \(i=0,1,2\)}
				\label{fig:resolution-real}
			\end{figure}
			
			\begin{figure}[!t]
				\begin{subfigure}{0.32\textwidth}
					\centering
					\includegraphics[width=\linewidth]{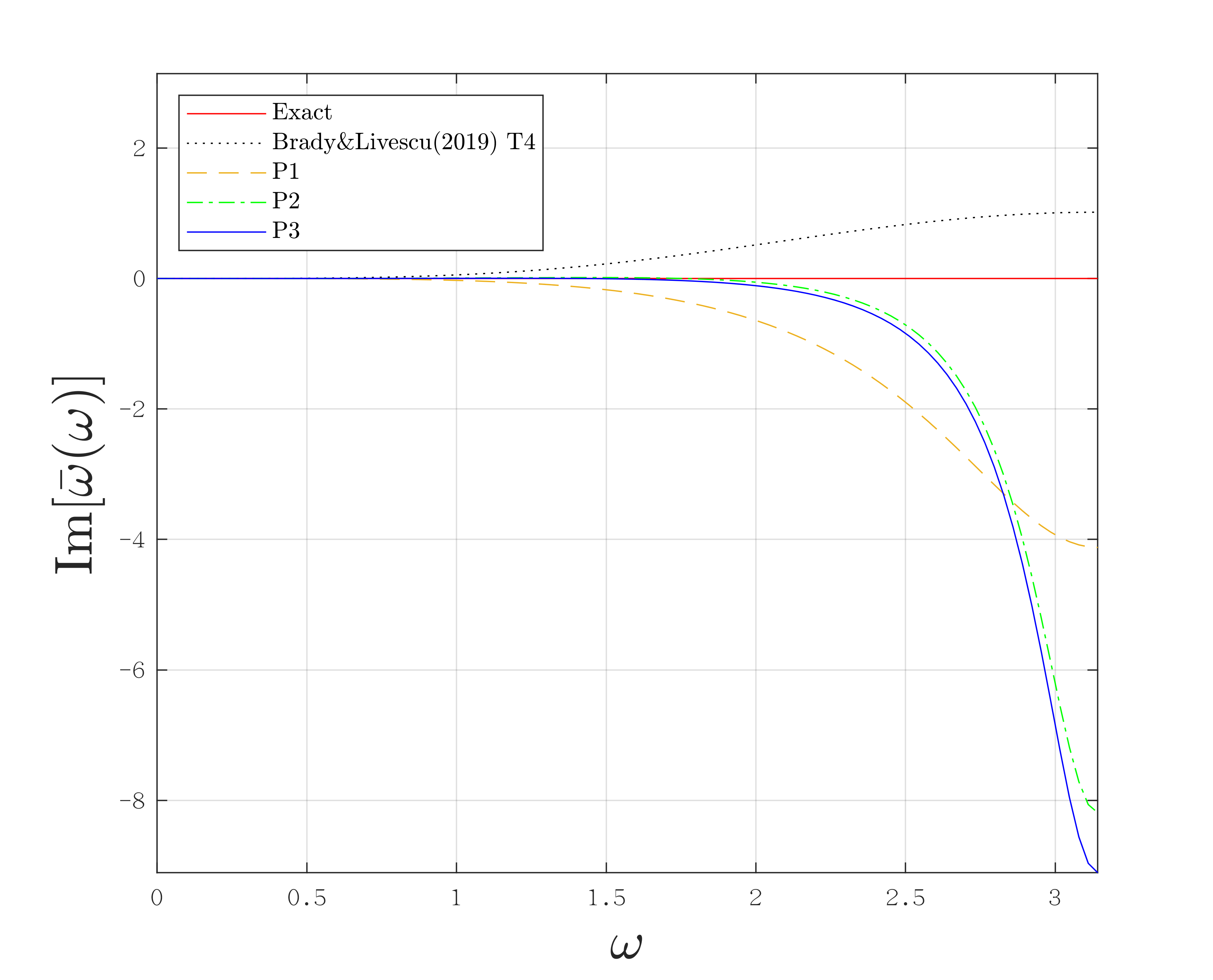}
					\caption{$i=0$}
					\label{fig:resolution-imag-i0}
				\end{subfigure}
				\begin{subfigure}{0.32\textwidth}
					\centering
					\includegraphics[width=\linewidth]{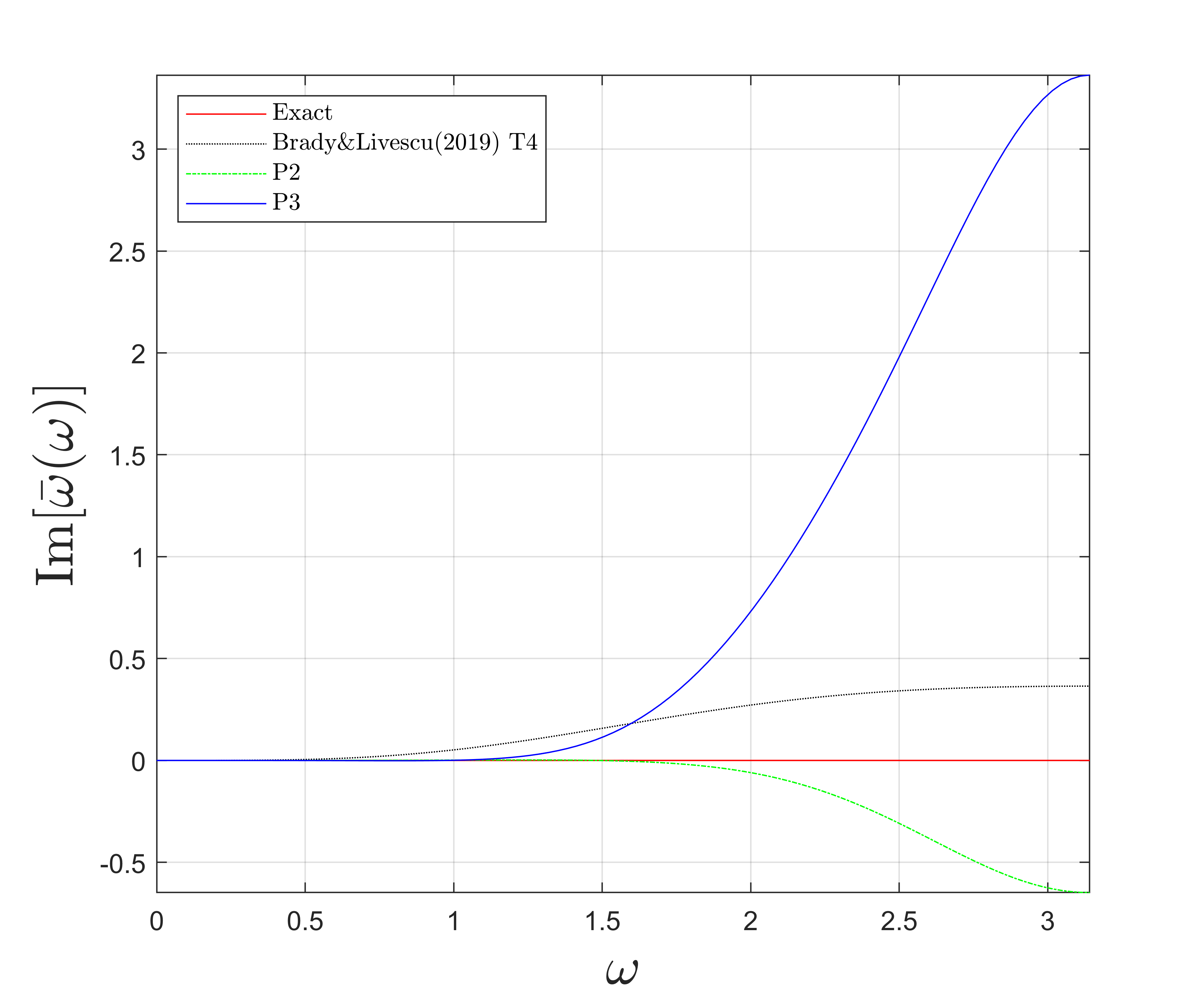}
					\caption{$i=1$}
					\label{fig:resolution-imag-i1}
				\end{subfigure}
				\begin{subfigure}{0.32\textwidth}
					\centering
					\includegraphics[width=\linewidth]{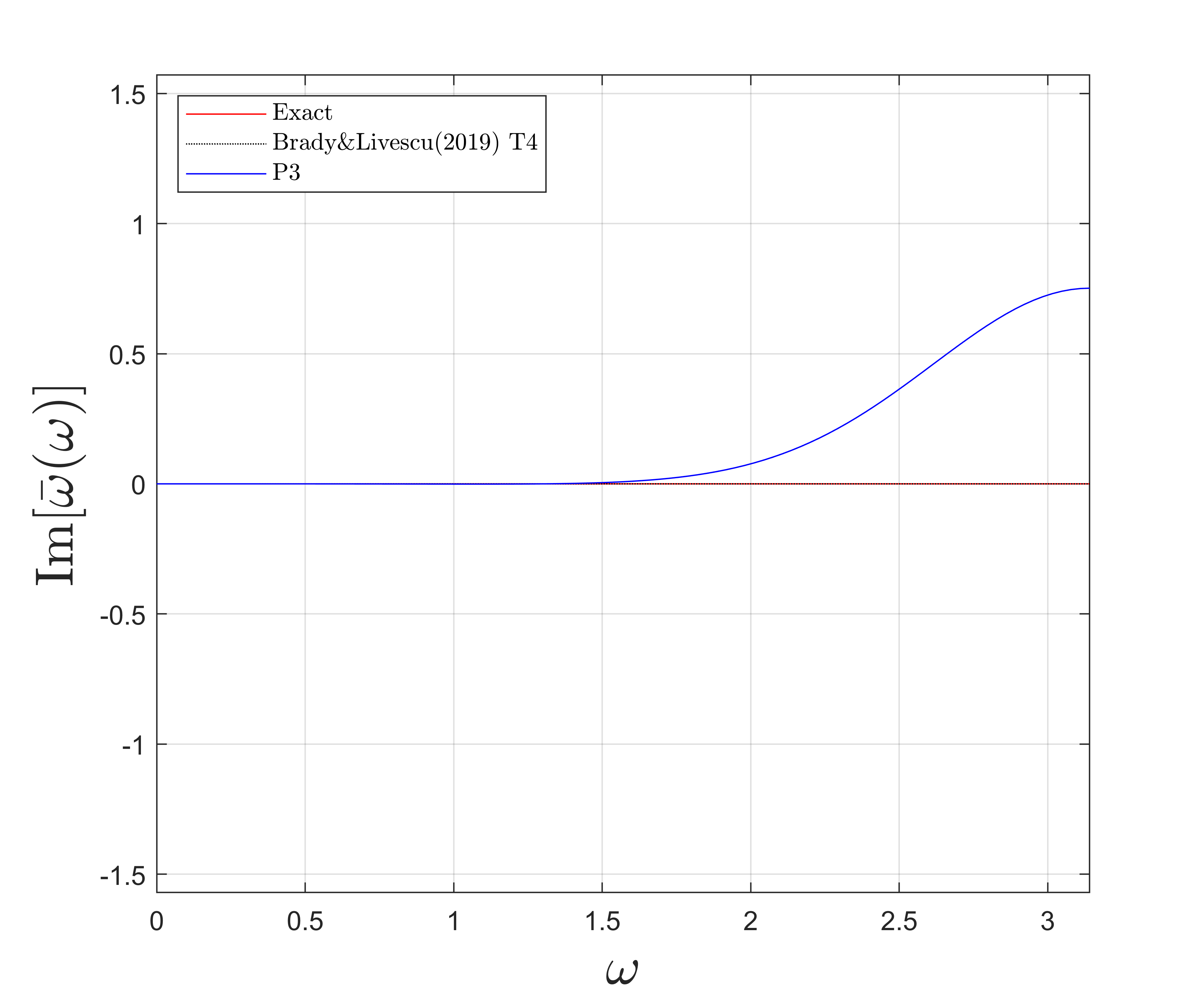}
					\caption{$i=2$}
					\label{fig:resolution-imag-i2}
				\end{subfigure}
				\caption{The imaginary parts of pseudo-wavenumber at nodes \(i=0,1,2\)}
				\label{fig:resolution-imag}
			\end{figure}

			\section{Numerical examples}\label{numerical}
			In this section, we conduct numerical experiments to test the accuracy orders of the three schemes, and demonstrate their performance in simulating complex fluids. 
			
			\subsection{One-dimensional linear advection}
			In this section, we use a simple one-dimensional linear advection equation to test the convergence order of the three schemes.
			Specifically, the equation is given by
			\begin{equation}
				\frac{\partial{u}}{\partial{t}} + \frac{\partial{u}}{\partial{x}} = 0,
			\end{equation}
			subject to the initial and boundary conditions:
			\begin{equation}
				u(x,0) = \sin(x), \quad u(0,t) = \sin(t).
			\end{equation}
			The exact solution is clearly
			\begin{equation}
				u(x,t) = \sin(x - t).
			\end{equation}
			
			To test the accuracy order using simulations, we employ progressively refined grid sizes($N = 65, 129, 257, 513, 1025$), and set the CFL condition $\Delta t / \Delta x = 0.5$ to ensure $L^2$ stability. 
			Temporal integration is performed using the fourth-order Runge-Kutta method, and all experiments are conducted over $t \in [0,1000]$.
			Convergence-order results are shown in Fig.\ref{fig:example1}. 
			It is evident that the schemes achieve the expected fourth-order accuracy.

			\begin{figure}[!t]
				\centering
				\includegraphics[scale=.5]{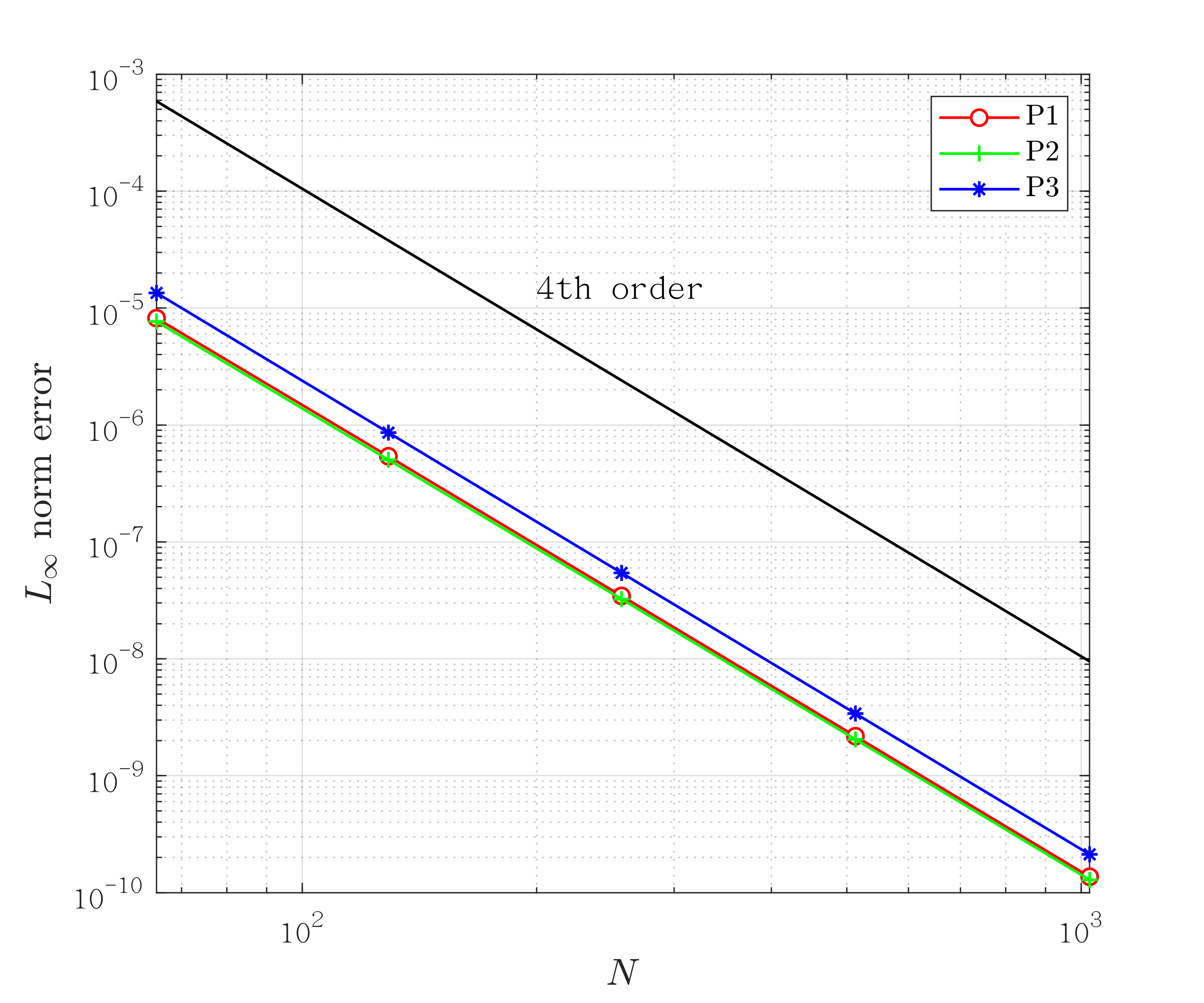}
				\caption{Maximum error in $u$ over $0\leq t\leq1000$, The expected order of each scheme is drawn as a solid black line.}
				\label{fig:example1}
			\end{figure}
			
			\subsection{Two-dimensional advection: non-homogeneous coefficients}
			In this section, we consider a two-dimensional spatial advection equation, where the coefficients (representing the wave speed in two dimensions) are both non-homogeneous. The equation reads
			\begin{equation}\label{2d_advection}
				\frac{\partial u}{\partial t} + c_x \frac{\partial u}{\partial x} + c_y \frac{\partial u}{\partial y} = 0, \quad 0 \leq x, y \leq L, \quad t \in [0,T],
			\end{equation}
			where the coefficients \( c_x \) and \( c_y \) are given as
			\begin{equation}
				c_x = \frac{\partial \psi}{\partial x}, \quad c_y = \frac{\partial \psi}{\partial y}, \quad \psi(x, y) = \sqrt{(x + 0.25)^2 + (y + 0.25)^2}.
			\end{equation}
			Eq. \eqref{2d_advection} is subject to the initial condition and boundary conditions:
			\begin{equation}
				u(x, y, 0) = \sin(2\pi \psi),
			\end{equation}
			and
			\begin{equation}
				\begin{aligned}
					u(0, y, t) &= \sin(2\pi (\psi(0, y) - t)), \\
					u(x, 0, t) &= \sin(2\pi (\psi(x, 0) - t)).
				\end{aligned}
			\end{equation}
			The exact solution to this equation is:
			\begin{equation}
				u(x, y, t) = \sin(2\pi (\psi - t)).
			\end{equation}
			
			In this experiment,  $L=\sqrt{2}$ and the number of grid points in each dimension is $N$, where $N\in\{ 21, 41, 61, 81\}$.
			To test numerical stability, we set $T = 1000$ and perform long-time integration about Eq. \eqref{2d_advection} over $t \in [0,1000]$.
			Results show that the long-time stability of the three schemes is demonstrated via the \( L_\infty \) error of the numerical solution versus time, as shown in {\bf Fig.~\ref{fig:example2}}.
			Clearly, {\bf Fig.~\ref{fig:example2}} confirms the long-time stability of the proposed schemes. 
			Additionally, we present the convergence orders of the three schemes in Fig.~\ref{fig:example2-ErrorRate}. 
			The results align with the theoretically predicted order.

			\begin{figure}[!t]
				\centering
				\begin{subfigure}{.32\textwidth}
					\centering
					\includegraphics[width=\linewidth]{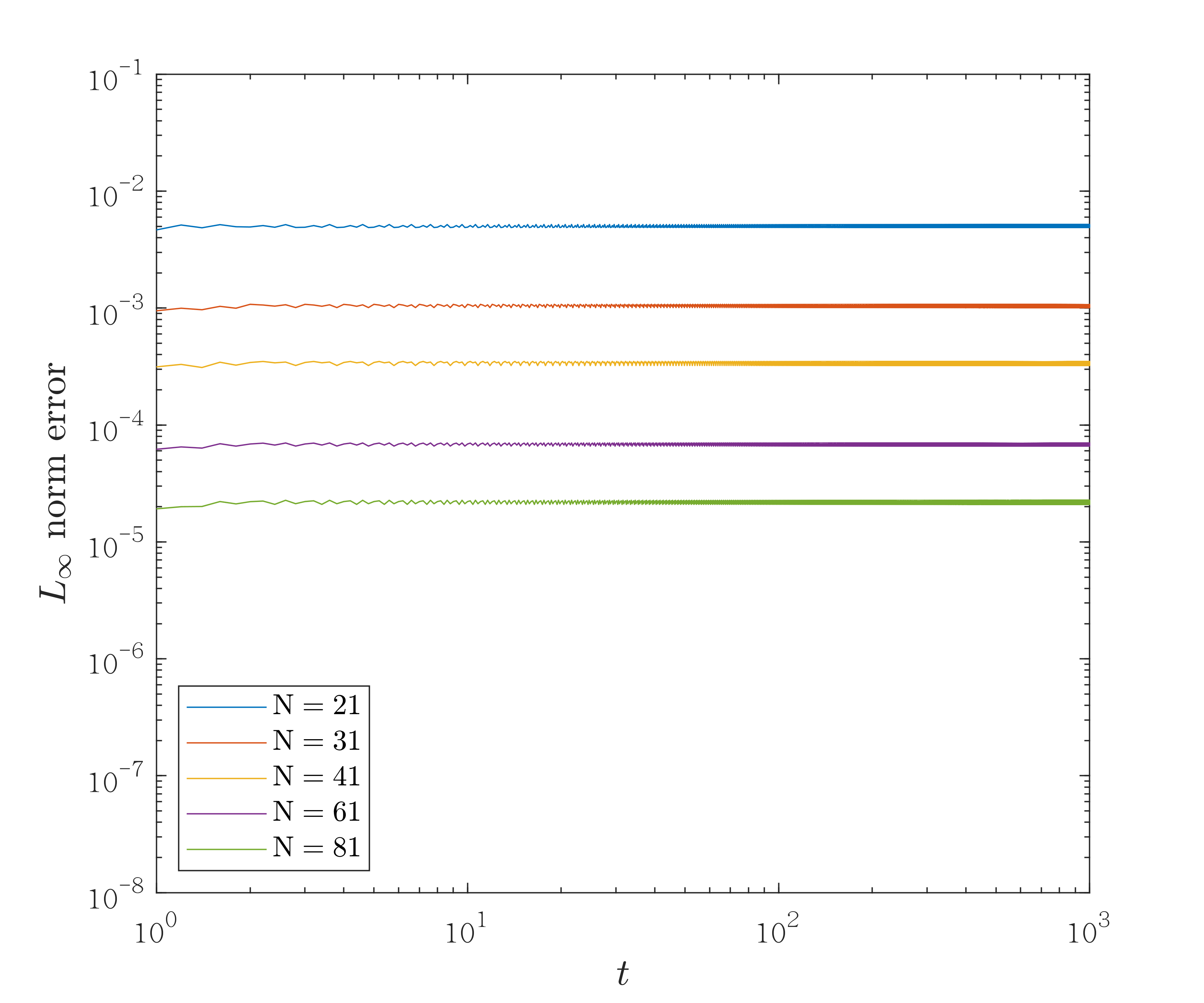}
					\caption{P1} 
				\end{subfigure}%
				\begin{subfigure}{.32\textwidth}
					\centering
					\includegraphics[width=\linewidth]{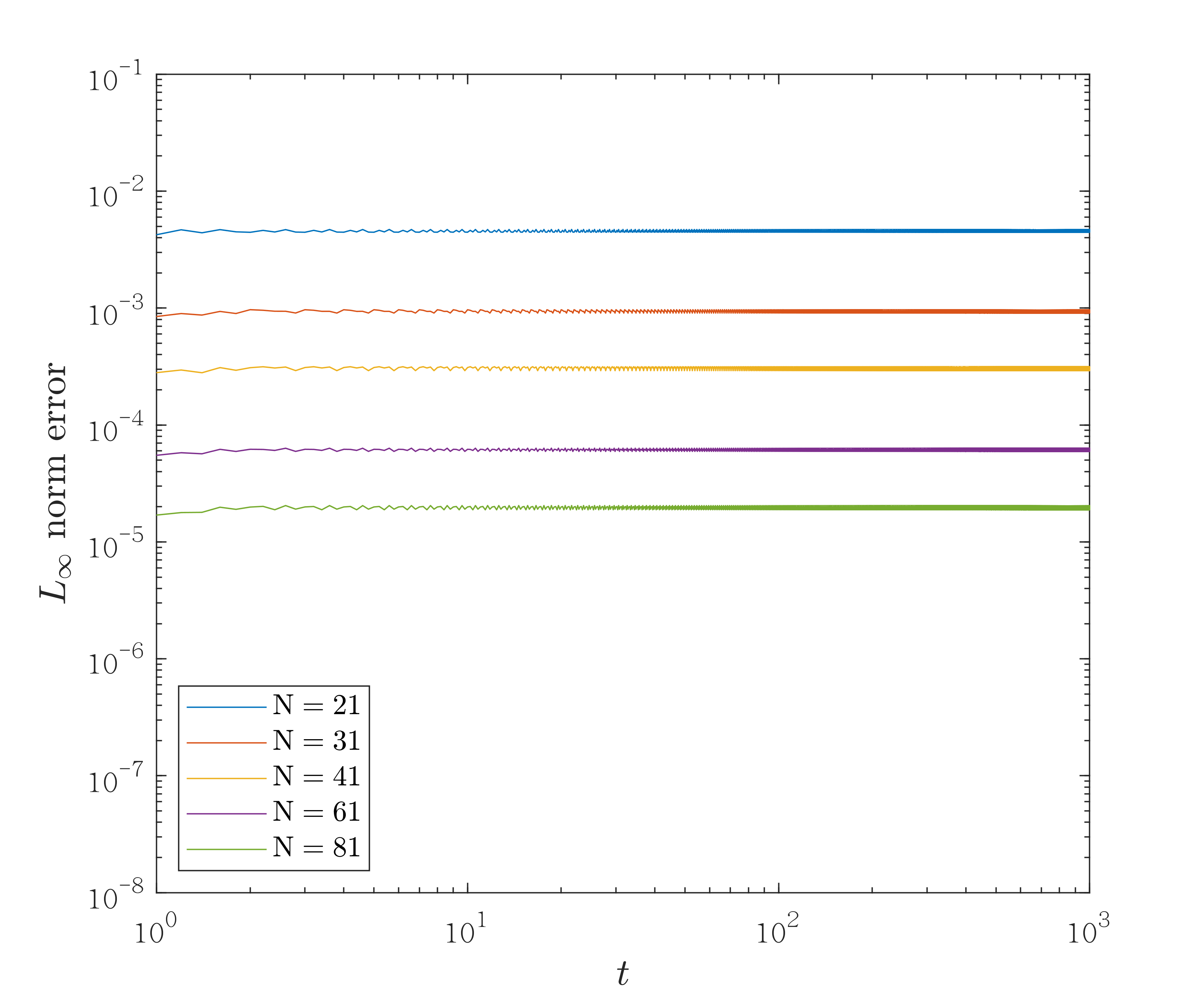}
					\caption{P2} 
				\end{subfigure}%
				\begin{subfigure}{.32\textwidth}
					\centering
					\includegraphics[width=\linewidth]{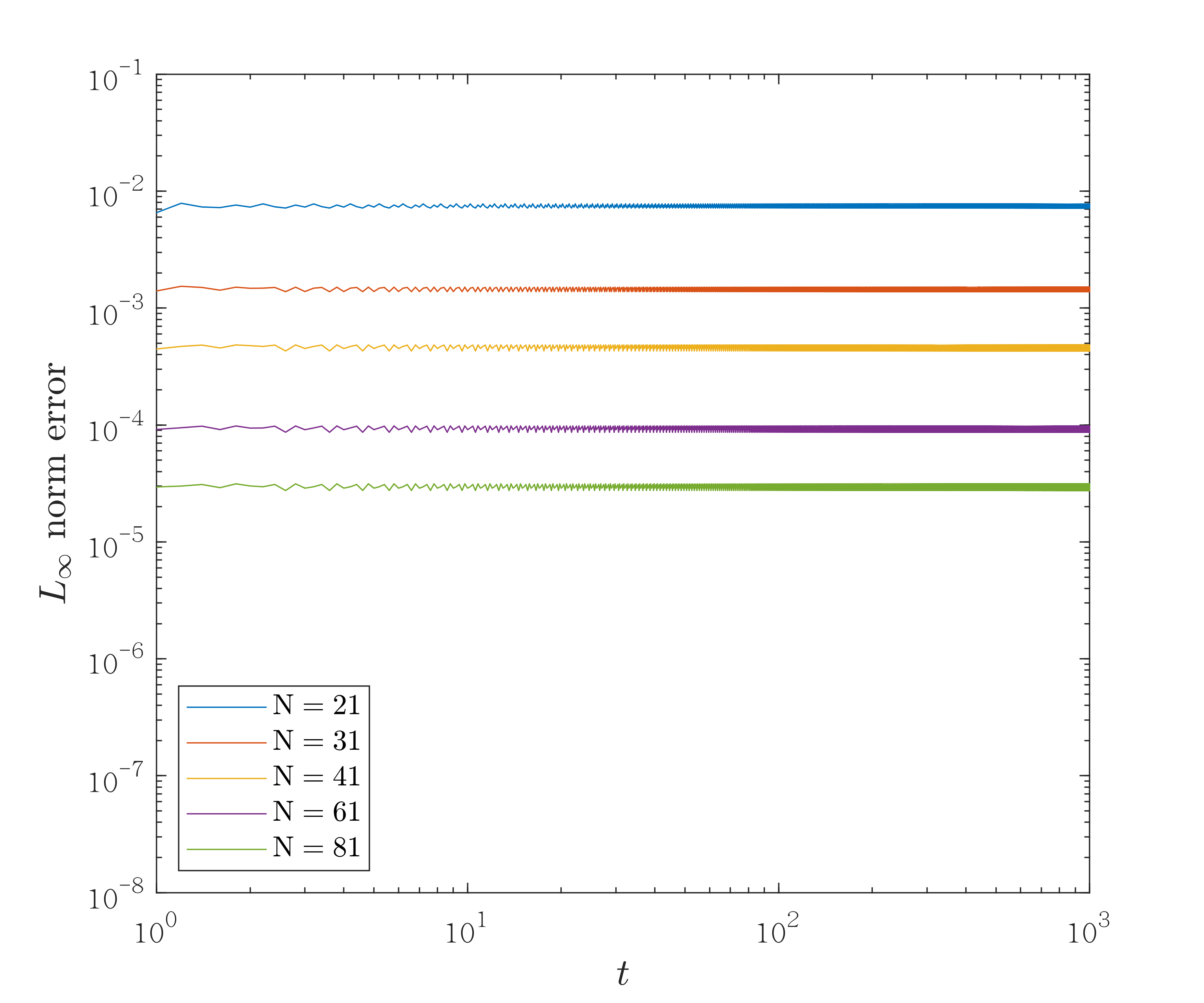}
					\caption{P3} 
				\end{subfigure}
				\caption{$L_\infty$ norm of error over $0\leq t\leq1000$, for indicated $N$ and constant timestep of $\Delta t=0.001$ for the varying coefficient scalar wave equation test.}
				\label{fig:example2}
			\end{figure}
			
			\begin{figure}[!t]
				\centering
				\includegraphics[scale=.5]{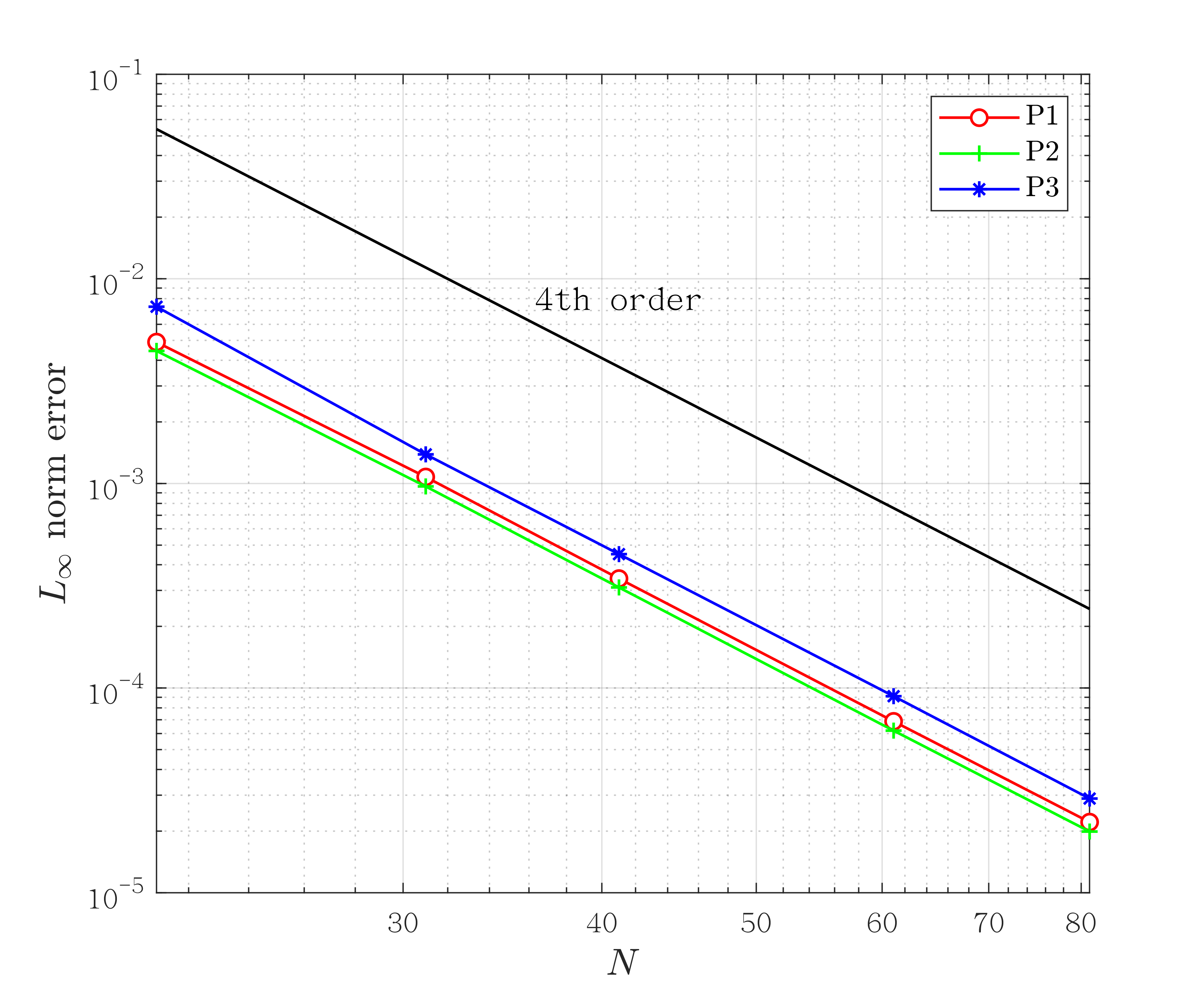}
				\caption{Maximum error in $u$ over $0\leq t\leq1000$, The expected order of each scheme is drawn as a solid black line.}
				\label{fig:example2-ErrorRate}
			\end{figure}
			
			\subsection{Two-dimensional inviscid vortex convection in a supersonic flow}
			In the last experiment, we consider the convection of a two-dimensional isentropic vortex in a supersonic free stream. 
			The governing equations are the compressible Euler equations:
			\begin{equation}
				\frac{\partial\mathbf{Q}}{\partial t} + \frac{\partial\mathbf{E}}{\partial x} + \frac{\partial\mathbf{F}}{\partial y} = \mathbf{0}, \quad x\in[-0.5L,L], y\in [-0.75L,0.75L],
			\end{equation}
			where
			\begin{equation}
				\mathbf{Q} = \begin{pmatrix}
					\rho \\
					\rho u \\
					\rho v \\
					\rho e_t
				\end{pmatrix}, \quad
				\mathbf{E} = \begin{pmatrix}
					\rho u \\
					\rho u^2 + p \\
					\rho uv \\
					(\rho e_t + p)u
				\end{pmatrix}, \quad
				\mathbf{F} = \begin{pmatrix}
					\rho v \\
					\rho uv \\
					\rho v^2 + p \\
					(\rho e_t + p)v
				\end{pmatrix}.
			\end{equation}
			Here, \(\rho\), $(u,v)$, and \(p\) represent the fluid density, flow velocity and the dynamic pressure, respectively. 
			$e_t$ denotes the total energy per unit mass,  which is defined as \(e_t = p / [(\gamma - 1)\rho] + (u^2 + v^2)/2\).
			In this computation, \(\gamma \) is set to $1.4$.
			The initial conditions are set as 
			\begin{equation}\label{woxuan}
				\begin{aligned}
					\frac{\rho(x, y)}{\rho_{\infty}} &= \left(1 - \frac{\gamma - 1}{2} \psi^2(x, y)\right)^{\frac{1}{\gamma - 1}}, \\
					\frac{u(x, y)}{a_{\infty}} &= M_{\infty} + K y \psi(x, y), \\
					\frac{v(x, y)}{a_{\infty}} &= -K x \psi(x, y), \\
					\frac{p(x, y)}{p_{\infty}} &= \left(\frac{\rho}{\rho_{\infty}}\right)^\gamma,
				\end{aligned}
			\end{equation}
			where
			\begin{equation}
				\psi(x, y) = \frac{\epsilon}{2\pi} \sqrt{\exp\left[1 - (x^2 + y^2) / R^2\right]}.
			\end{equation}
			Here, $\epsilon \in \{0.1,1.5,4\}$, where $\epsilon = 0.1$ stands for the linear case, and higher values represent nonlinear cases.
			$R$ stands for the size of the vortex, which is set to $0.08L$ initialy. 
			The free stream velocity is given by \(u_\infty = M_\infty a_\infty\), where \(a_\infty = \sqrt{\gamma p_\infty / \rho_{\infty}}\) denotes the speed of sound, and \(M_\infty\) denotes the Mach number.
			The exact solution is obtained by substituting \(\hat{x} = x - u_\infty t\) into equation Eq. \eqref{woxuan}.
			Supersonic inflow boundary conditions are applied at the domain inlet in the free stream direction, while no boundary conditions are applied at the exit boundary. Periodic boundary conditions are applied in the \(y\)-direction.
			
			To perform the simulation, a uniform grid is used to discretize the entire spatial domain, with the number of grid points in the $x-$ and $y-$directions both set to $150$.
			The CFL number is set to $0.5$, and the time-step $\Delta t$ is set to $0.5 h_x / u_\infty$ to meet the CFL condition. 
			To test the long time stability of the numerical simulation, we simulate the vortex function $\omega$, which is related with the velocity via $\omega = \partial v / \partial x - \partial u / \partial y$, over a long time interval. 
			During this interval, the vortex advects from the left boundary, travels through the entire domain, arrives at the right boundary, and finally leaves the domain.
			Some snapshots of the advection process are shown in Fig.\ref{fig:example3-phase}.
			Clearly, the vortex advects steadily and preserves its shape well in the long-time simulation, indicating the stability of the numerical schemes P1, P2, and P3. 
			Fig.\ref{fig:example3-error} shows the pressure error over the time interval $[0 \leq t u_\infty / L \leq 200]$ for different values of $\epsilon$.
			Obviously, the numerical error decreases rapidly, especially for scheme P3, indicating the fast convergence of all proposed numerical schemes.
			Finally, we also present the convergence order of the three schemes in {\bf Fig.~\ref{fig:example3-ErrorRate}}.
			It is evident that all schemes achieve the expected convergence order, demonstrating the success of our proposed numerical schemes.
			
			\begin{figure}[!t]
				\centering
				\begin{subfigure}{.32\textwidth}
					\centering
					\includegraphics[width=\linewidth]{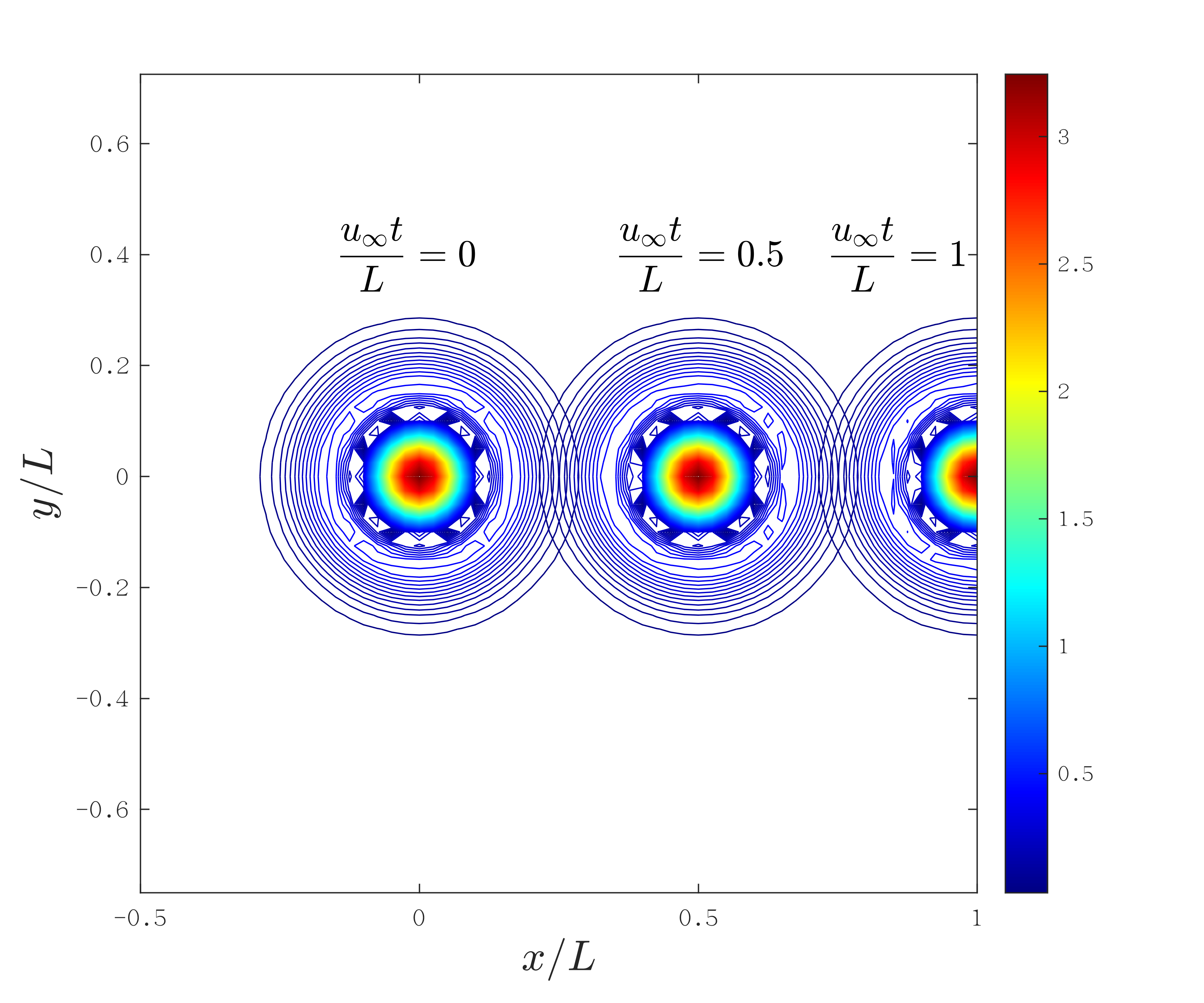}
					\caption{$P1$} 
				\end{subfigure}%
				\begin{subfigure}{.32\textwidth}
					\centering
					\includegraphics[width=\linewidth]{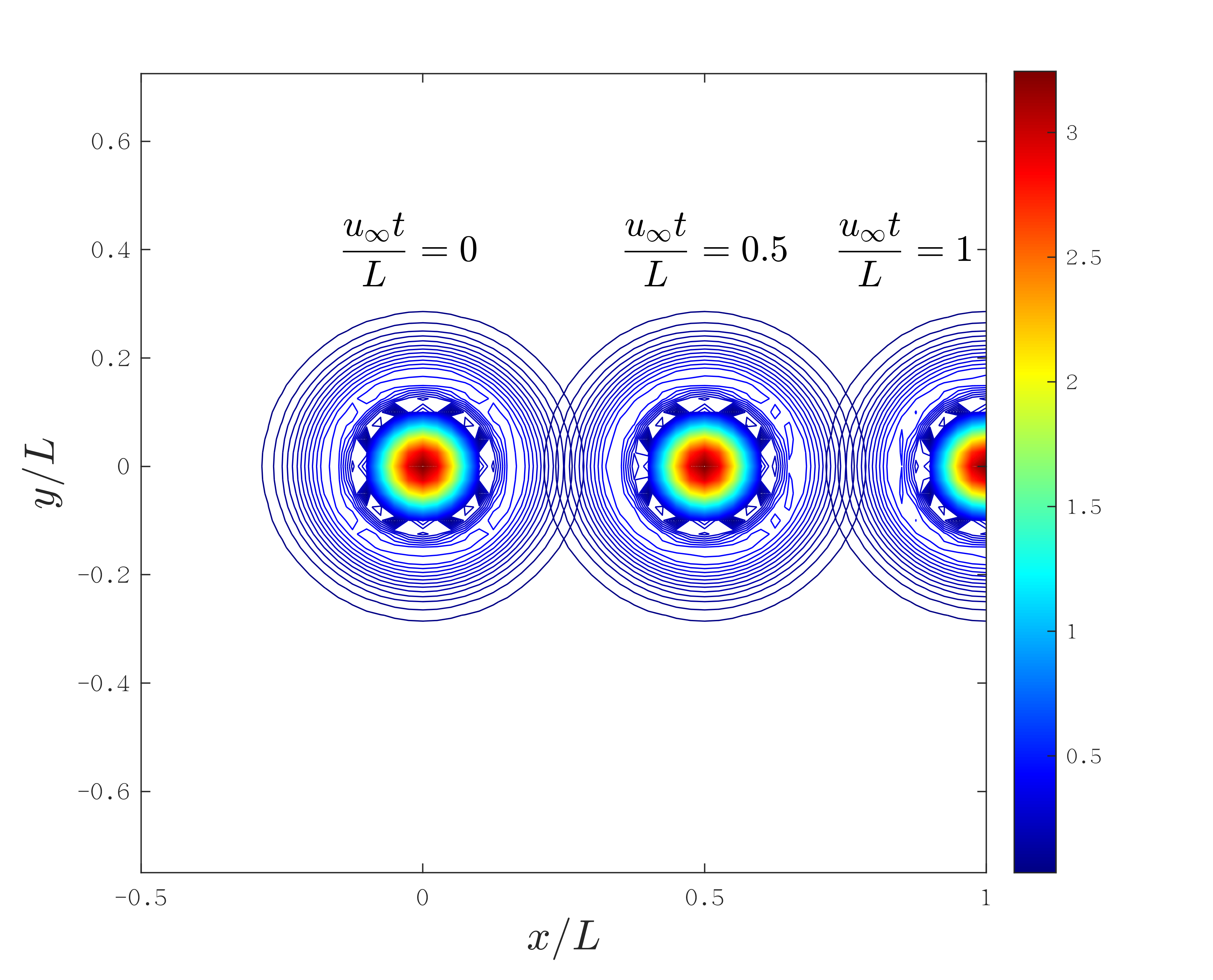}
					\caption{$P2$} 
				\end{subfigure}%
				\begin{subfigure}{.32\textwidth}
					\centering
					\includegraphics[width=\linewidth]{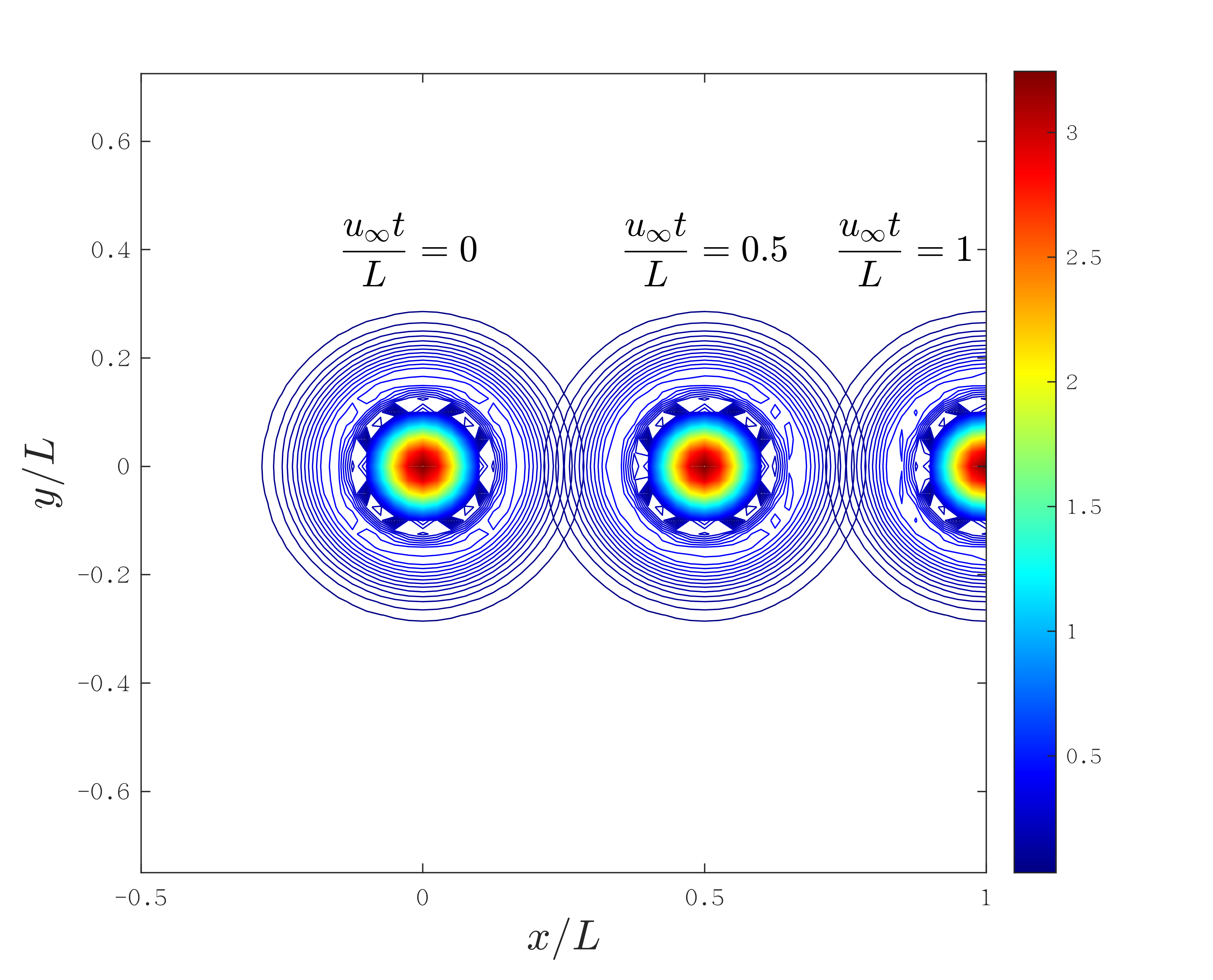}
					\caption{$P3$} 
				\end{subfigure}
				\caption{Two-dimensional inviscid vortex convection calculated by the present schemes. Contours of normalised absolute vorticity: $|\omega|L/(u_\infty\varepsilon)$. $\varepsilon=0.1$, $(N\times N)=(60,60)$.}
				\label{fig:example3-phase}
			\end{figure}

			\begin{figure}[!t]
				\centering
				\begin{subfigure}{.32\textwidth}
					\centering
					\includegraphics[width=\linewidth]{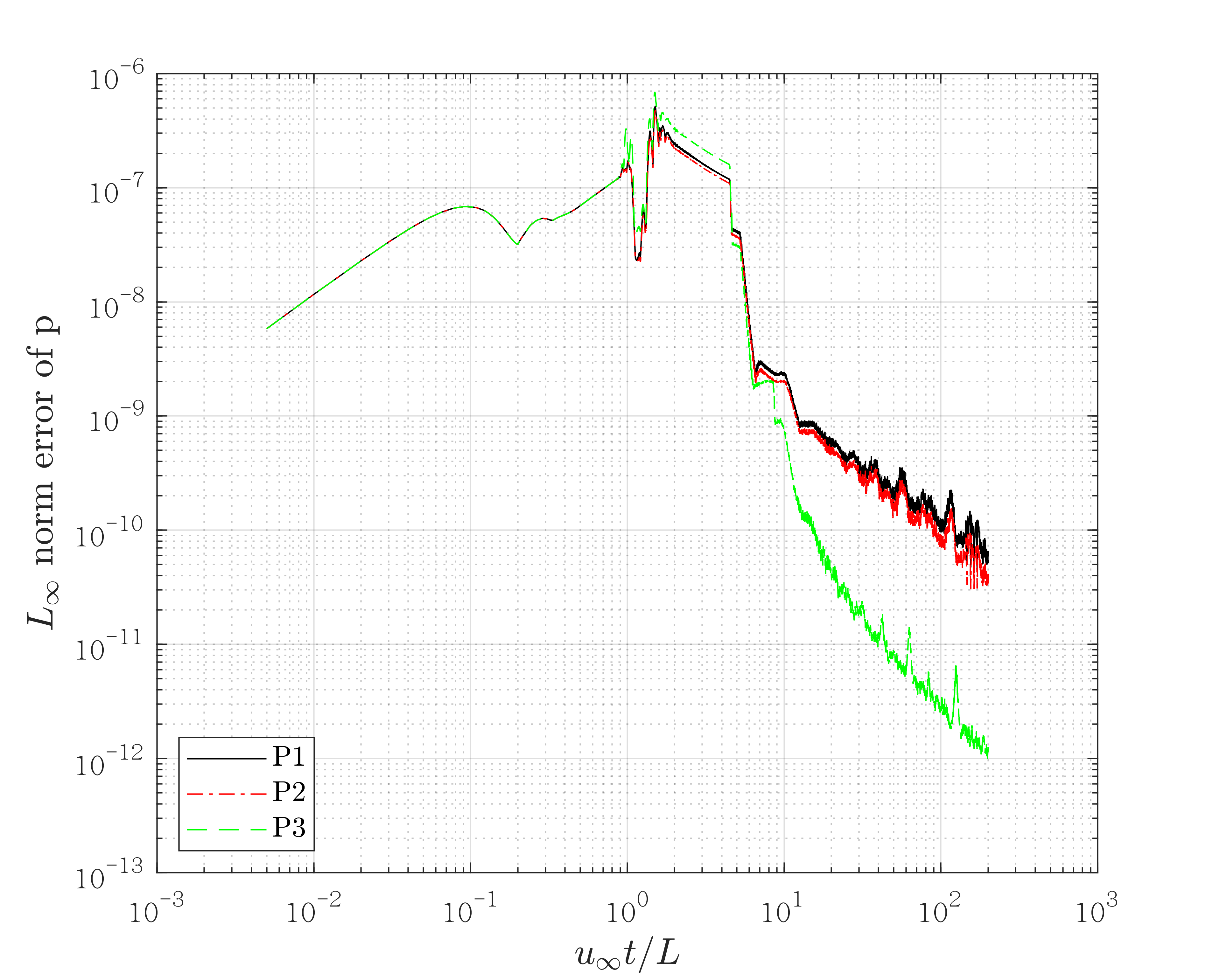}
					\caption{$\varepsilon=0.1$} 
				\end{subfigure}%
				\begin{subfigure}{.32\textwidth}
					\centering
					\includegraphics[width=\linewidth]{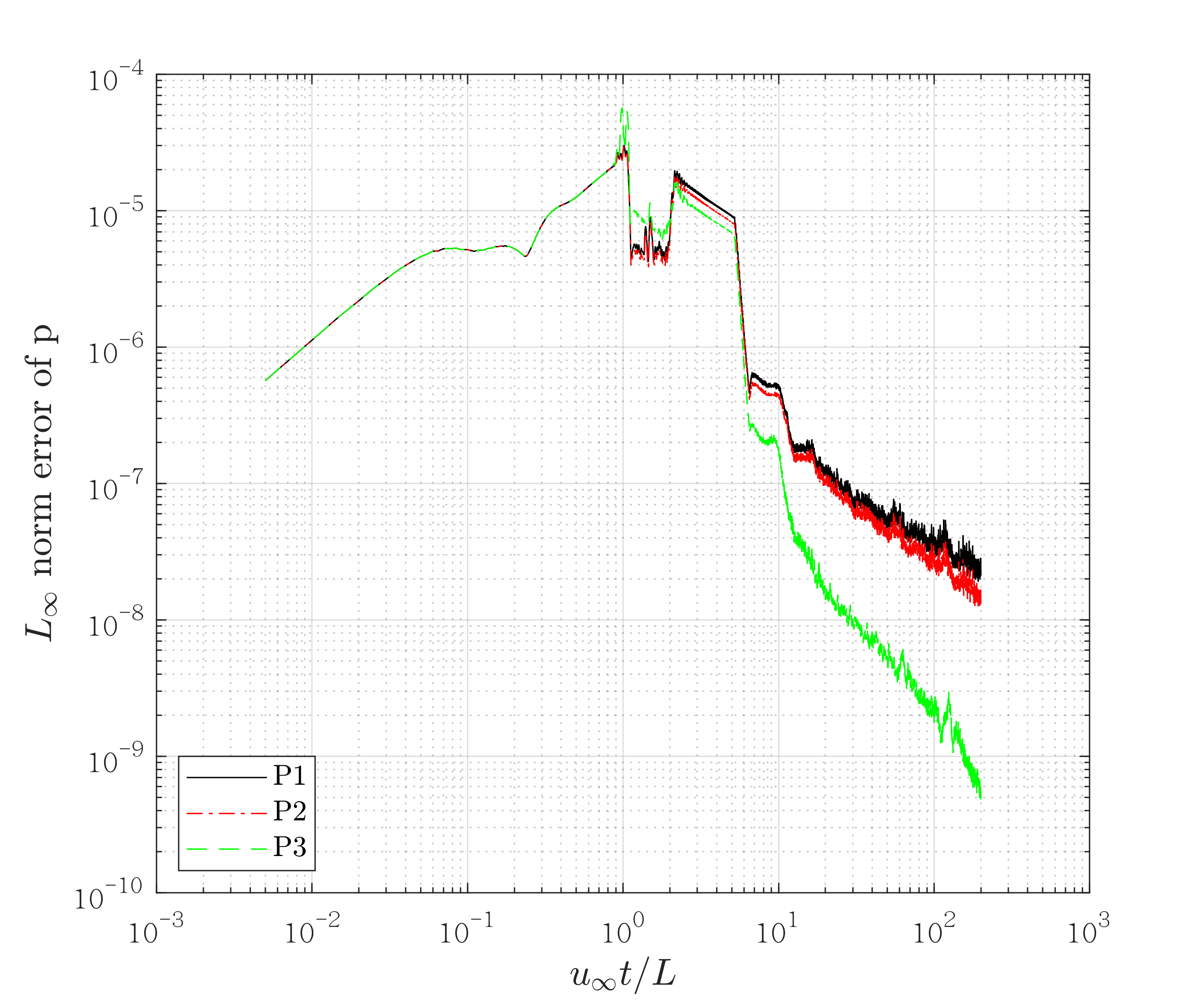}
					\caption{$\varepsilon=1.5$} 
				\end{subfigure}%
				\begin{subfigure}{.32\textwidth}
					\centering
					\includegraphics[width=\linewidth]{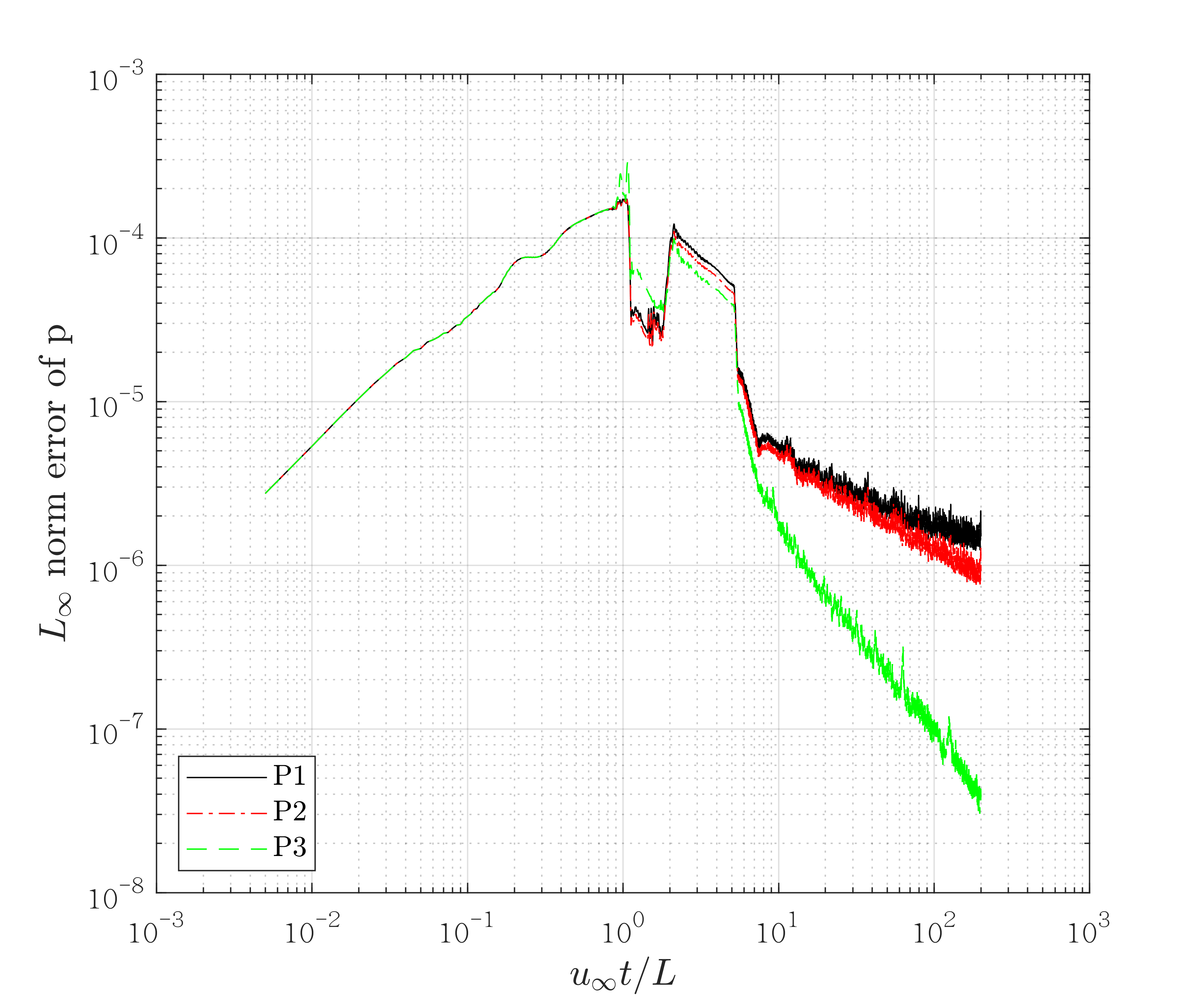}
					\caption{$\varepsilon=4.0$} 
				\end{subfigure}
				\caption{Time history of the $L_\infty$ norm errors in the calculations of the two-dimension inviscid convection problem for $\epsilon=0.1,\epsilon=1.5$ and $\epsilon=4$.}
				\label{fig:example3-error}
			\end{figure}

			\begin{figure}[!t]
				\centering
				\begin{subfigure}{.32\textwidth}
					\centering
					\includegraphics[width=\linewidth]{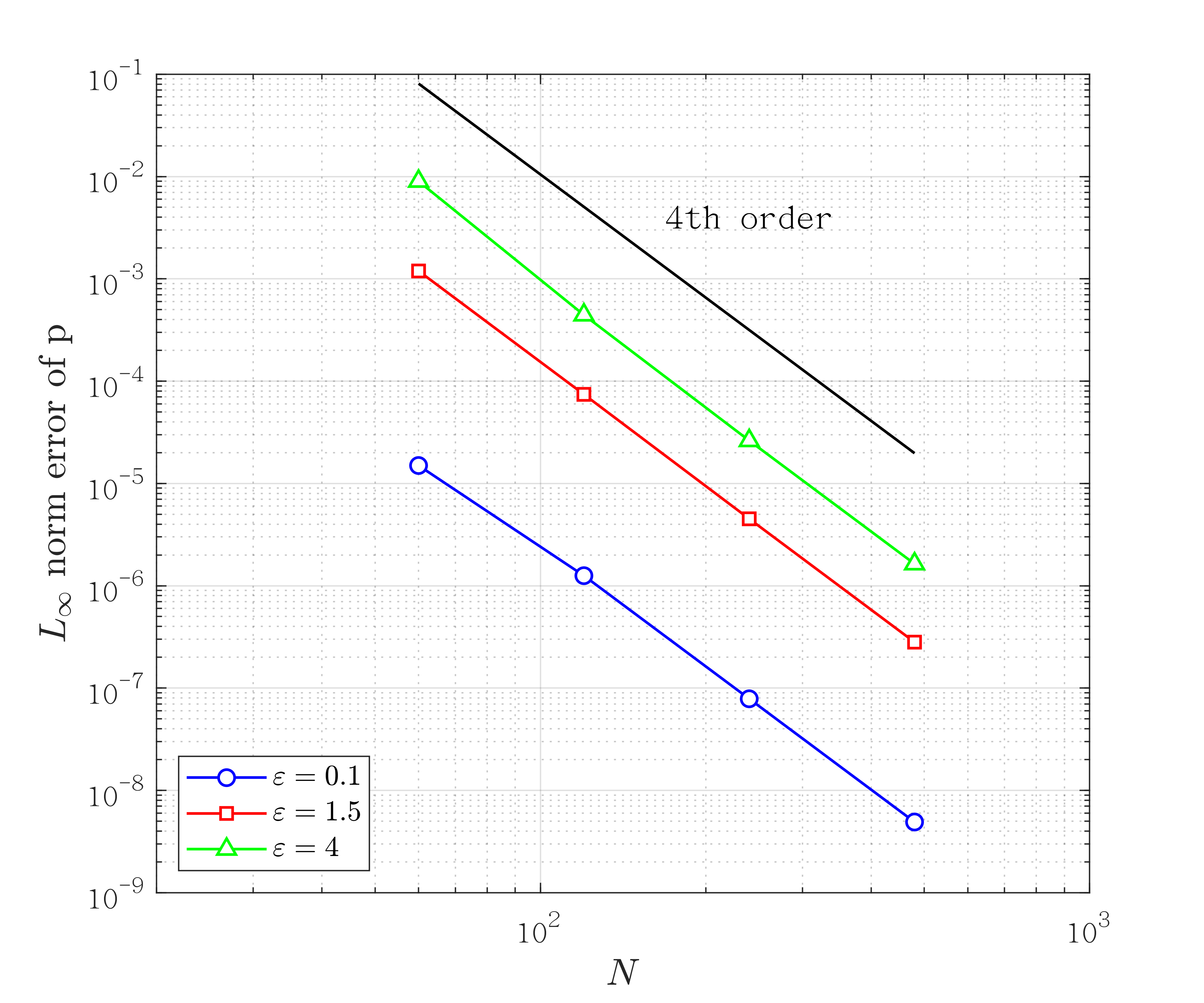}
					\caption{P1} 
				\end{subfigure}%
				\begin{subfigure}{.32\textwidth}
					\centering
					\includegraphics[width=\linewidth]{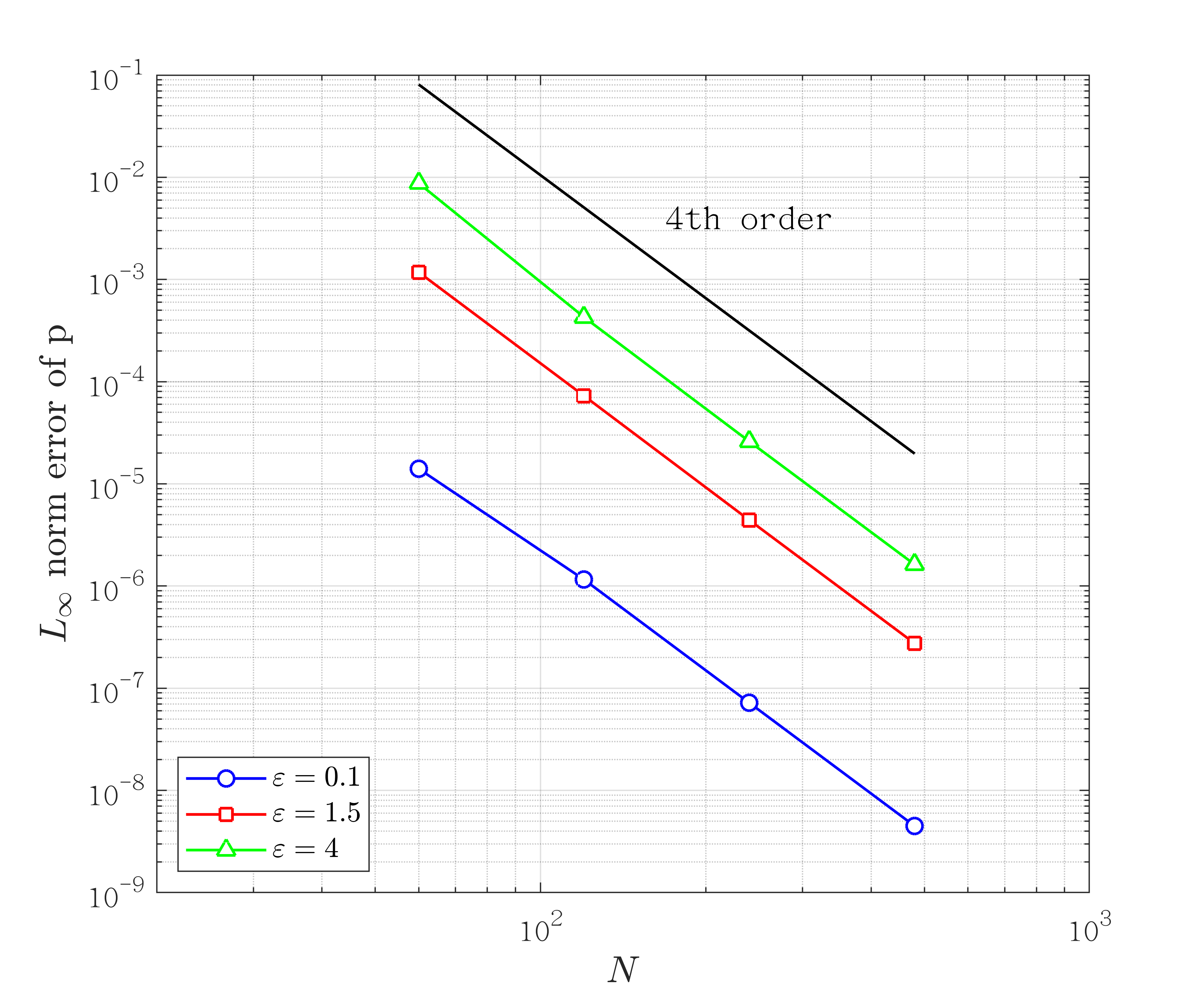}
					\caption{P2} 
				\end{subfigure}%
				\begin{subfigure}{.32\textwidth}
					\centering
					\includegraphics[width=\linewidth]{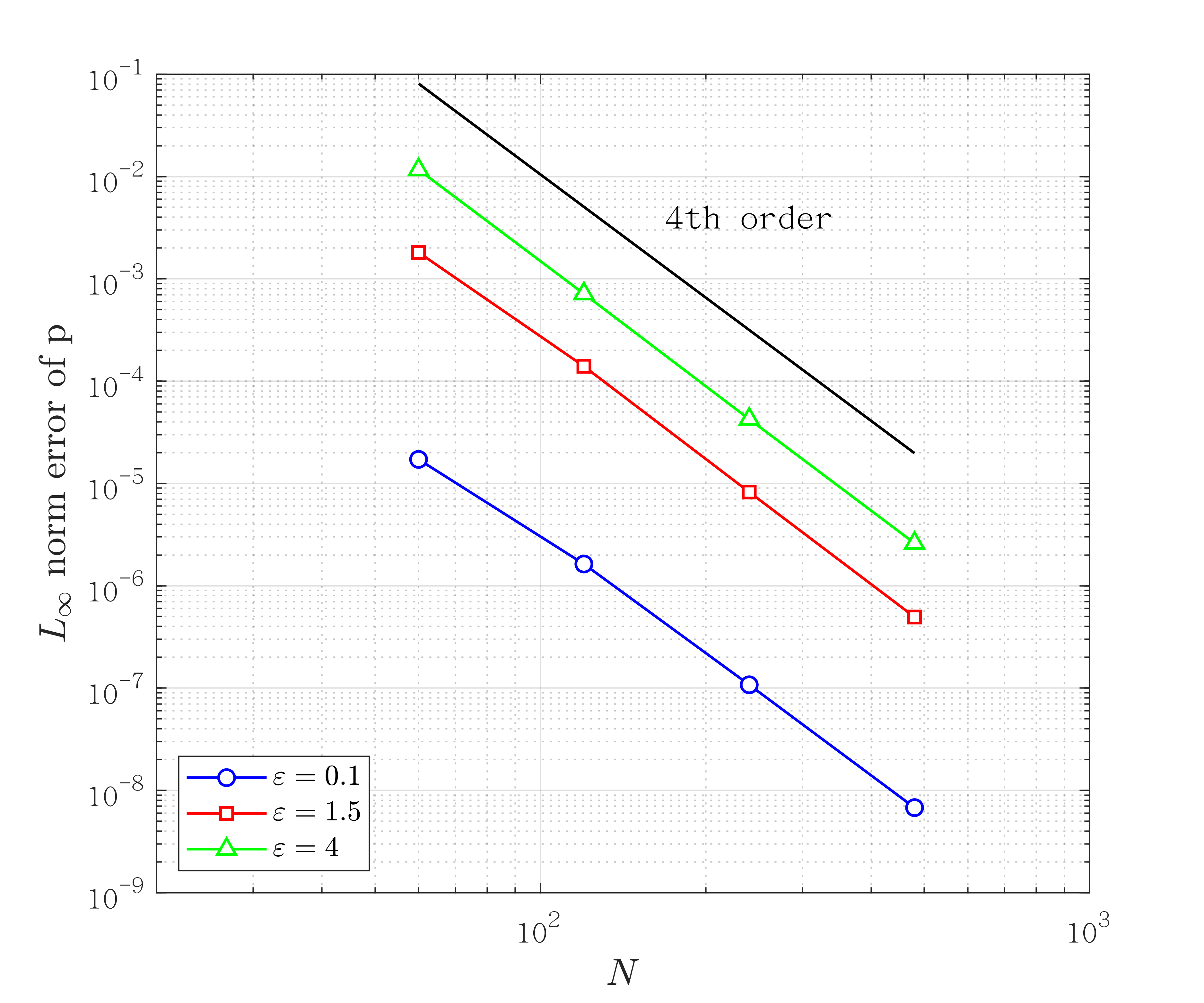}
					\caption{P3} 
				\end{subfigure}
				\caption{Maximum error in $p$ over $0\leq u_\infty t/L\leq 200$, The expected order of each scheme is drawn as a solid black line.}
				\label{fig:example3-ErrorRate}
			\end{figure}
			
			\section{Conclusion}\label{conclusion}
			In this paper, we present a novel globally conservative compact finite difference framework for hyperbolic conservation laws, which advances the boundary treatment methodology proposed by Lele.
			Within this framework, the highest algebraic precision is attained when discretizing the continuous conservation law through numerical integration into its discrete form.
			Fourth-order conservative schemes are developed under this framework, with optimized boundary schemes designed to enhance resolution and asymptotic stability.
			Numerical experiments are conducted to verify the convergence order and long-term stability of the proposed schemes. In our subsequent research, we will develop conservative hybrid compact difference schemes for discontinuous problems.

\section*{Acknowledgments}
The work of Z. Sun was supported in part by NSFC (No.11901496) and the Project of Hunan Provincial Education Department of China (No.24B0168). The work of W. Yao was supported in part by NSFC (No.12471378) and the Natural Science Foundation of Guangdong Province of China (No.2024A1515010356). The work of L. Wang was supported in part by the Hunan Provincial Natural Science Foundation (No.2022JJ40117).


\newpage
\section{Appendix}

\subsection{The necessary condition for the construction of schemes P1, P2 and P3}\label{app:necessary}

Scheme P1 must satisfy Eq. \eqref{taylor:i=0}, along with two additional conditions:
\begin{equation}\label{bsh1-P1}
	\left\{
	\begin{aligned}
		w'_0a_{00} &= w_0 - \frac{1}{6},\\
		w'_0a_{01} &= w_1 - \frac{5}{6},\\
		w'_0a_{02} &= w_2-1,\\
		w'_0a_{03} &= w_3-1.
	\end{aligned}
	\right.
\end{equation}
\begin{equation}\label{bsh2-P1}
	\left\{
	\begin{aligned}
		w'_0b_{00} &= -\frac{1}{2},\\
		w'_0b_{01} &= \frac{1}{2},\\
		w'_0b_{02} &= 0,\\
		w'_0b_{03} &= 0.
	\end{aligned}
	\right.
\end{equation}
 Scheme P2 must satisfy the constraints in Eqs. \eqref{taylor:i=0} and \eqref{taylor:i=1}, as well as those in Eqs. \eqref{bsh1-P2} and \eqref{bsh2-P2}.
\begin{equation}\label{bsh1-P2}
	\left\{
	\begin{aligned}
		w'_0a_{00}+w'_1a_{10} &= w_0,\\
		w'_0a_{01}+w'_1a_{11} &= w_1 - \frac{1}{6},\\
		w'_0a_{02}+w'_1a_{12} &= w_2 - \frac{5}{6},\\
		w'_0a_{03}+w'_1a_{13} &= w_3-1.
	\end{aligned}
	\right.
\end{equation}
\begin{equation}\label{bsh2-P2}
	\left\{
	\begin{aligned}
		w'_0b_{00}+w'_1b_{10} &= -1,\\
		w'_0b_{01}+w'_1b_{11} &= \frac{1}{2},\\
		w'_0b_{02}+w'_1b_{12} &= \frac{1}{2},\\
		w'_0b_{03}+w'_1b_{13} &= 0.
	\end{aligned}
	\right.
\end{equation}
 For scheme P3, the constraint Eqs. \eqref{taylor:i=0},  \eqref{taylor:i=1},  \eqref{taylor:i=2}, \eqref{bsh1} and \eqref{bsh2} must be satisfied, which has been given in the body part.
 
 \begin{lemma}
 For schemes P1, P2 and P3, if the corresponding set of equations has a solution, then the following relationship will hold
\begin{equation}\label{w0w1w2w3}
	\begin{aligned}
		w_1 &= -3w_0+\frac{55}{24},\\
		w_2 &= 3w_0-\frac{1}{6},\\
		w_3 &= -w_0+\frac{11}{8},
	\end{aligned}
\end{equation}
which is a common condition for the there shemes P1, P2 and P3.
\end{lemma}
\begin{proof}
For the convenience of discussion, we rewrite Eq. \eqref{taylor:i=0}-Eq. \eqref{taylor:i=2} as 
\begin{equation}\label{relation_Ai_Bi}
	A_i \vec{a}_i = B_i \vec{b}_i, \quad  i =0,1,2,\end{equation}
where $\vec{a}_i = (a_{i0},a_{i1},a_{i2},a_{i3})^\top$ and  $\vec{b}_i = (b_{i0},b_{i1},b_{i2},b_{i3})^\top$, and $A_i$, $B_i$ are coefficent matrices deirved from Eq. \eqref{taylor:i=0}-Eq. \eqref{taylor:i=2}.
It is easy to check that 
\begin{equation} \label{coe-mat-relation}
	B_0^{-1} A_0 =B_1^{-1} A_1 =B_2^{-1} A_2 = 
	\begin{pmatrix}
		-\frac{11}{6}    &       -\frac{1}{3}     &       \frac{1}{6}    &       -\frac{1}{3} \\
		3       &      -\frac{1}{2}       &    -1       &       \frac{3}{2}     \\
		-\frac{3}{2}    &        1        &      \frac{1}{2}     &      -3       \\
		\frac{1}{3}     &      -\frac{1}{6}       &     \frac{1}{3}      &     \frac{11}{6}  
	\end{pmatrix} \triangleq C.
\end{equation}	

For scheme P1
\begin{equation}
	A=\begin{bmatrix}
		a_{00}&a_{01}&a_{02}&a_{03}&\cdots&0&0&0&0\\
		\frac{1}{6}&\frac{2}{3}&\frac{1}{6}&0&0&\cdots&0&0&0\\
		0&\frac{1}{6}&\frac{2}{3}&\frac{1}{6}&0&0&\cdots&0&0\\
		0&0&\frac{1}{6}&\frac{2}{3}&\frac{1}{6}&0&0&\cdots&0\\
		\vdots&\ddots&\ddots&\ddots&\ddots&\ddots&\ddots&\ddots&\vdots\\
		0&\cdots&0&0&\frac{1}{6}&\frac{2}{3}&\frac{1}{6}&0&0\\
		0&0&\cdots&0&0&\frac{1}{6}&\frac{2}{3}&\frac{1}{6}&0\\
		0&0&0&\cdots&0&0&\frac{1}{6}&\frac{2}{3}&\frac{1}{6}\\
		0&0&0&0&\cdots&a_{03}&a_{02}&a_{01}&a_{00}
	\end{bmatrix}
\end{equation}
and
\begin{equation}
	B=\begin{bmatrix}
		b_{00}&b_{01}&b_{02}&b_{03}&\cdots&0&0&0&0\\
		-\frac{1}{2}&0&\frac{1}{2}&0&0&\cdots&0&0&0\\
		0&-\frac{1}{2}&0&\frac{1}{2}&0&0&\cdots&0&0\\
		0&0&-\frac{1}{2}&0&\frac{1}{2}&0&0&\cdots&0\\
		\vdots&\ddots&\ddots&\ddots&\ddots&\ddots&\ddots&\ddots&\vdots\\
		0&\cdots&0&0&-\frac{1}{2}&0&\frac{1}{2}&0&0\\
		0&0&\cdots&0&0&-\frac{1}{2}&0&\frac{1}{2}&0\\
		0&0&0&\cdots&0&0&-\frac{1}{2}&0&\frac{1}{2}\\
		0&0&0&0&\cdots&-b_{03}&-b_{02}&-b_{01}&-b_{00}
	\end{bmatrix}.
\end{equation}
In addition, according to Eq. \eqref{wp_relation_w}, 
\begin{equation} 
	\vec{a}_0 \cdot 
	w_0'
	= \begin{pmatrix}
		w_0-\frac{1}{6} \\  w_1-\frac{5}{6} \\ w_2-1 \\ w_3- 1
	\end{pmatrix},
\end{equation}
\begin{equation}\label{bsh2-matrix-form-P1} 
	\vec{b}_0  
	\cdot
	w_0'
	= 
	\begin{pmatrix}
		-\frac{1}{2} \\  \frac{1}{2} \\ 0 \\ 0
	\end{pmatrix}.
\end{equation} 
The other weights can be set $w_i'=1,i=1,\dots,N-1$ and $w_i=1,i=4,\dots,N-4$. Using Eq. \eqref{relation_Ai_Bi}, and noticing Eq. \eqref{coe-mat-relation}, one rewrites Eq. \eqref{bsh2-matrix-form-P1} as
\begin{equation}
	C
	\begin{pmatrix}
		w_0-\frac{1}{6} \\  w_1-\frac{5}{6} \\ w_2-1 \\ w_3- 1
	\end{pmatrix}
	= 
	\begin{pmatrix}
		-\frac{1}{2} \\  \frac{1}{2} \\ 0 \\ 0
	\end{pmatrix},
\end{equation} 
This follows from 
\begin{equation}
	C
	\begin{pmatrix}
		w_0 \\  w_1 \\ w_2 \\ w_3
	\end{pmatrix}
	= 
	\begin{pmatrix}
		-\frac{1}{2} \\ \frac{1}{2} \\ 0 \\ 0
	\end{pmatrix}
	- C
	\begin{pmatrix}
		-\frac{1}{6} \\  -\frac{5}{6} \\ -1 \\ -1
	\end{pmatrix}.
\end{equation} 
Computing the right-hand side, we have
\begin{equation}
	\begin{pmatrix}
		-\frac{1}{2} \\  \frac{1}{2} \\ 0 \\ 0
	\end{pmatrix}
	-
	C
	\begin{pmatrix}
		-\frac{1}{6} \\  -\frac{5}{6} \\ -1 \\ -1
	\end{pmatrix}
	= 
	\begin{pmatrix}
		-\frac{1}{2} \\  \frac{1}{2} \\ 0 \\ 0
	\end{pmatrix}
	-
	\begin{pmatrix}
		-\frac{11}{6}    &       -\frac{1}{3}     &       \frac{1}{6}    &       -\frac{1}{3} \\
		3       &      -\frac{1}{2}       &    -1       &       \frac{3}{2}     \\
		-\frac{3}{2}    &        1        &      \frac{1}{2}     &      -3       \\
		\frac{1}{3}     &      -\frac{1}{6}       &     \frac{1}{3}      &     \frac{11}{6}  
	\end{pmatrix}
	\begin{pmatrix}
		-\frac{1}{6} \\  -\frac{5}{6} \\ -1 \\ -1
	\end{pmatrix}
	= 
	\begin{pmatrix}
		-\frac{5}{4} \\  \frac{13}{12} \\ -\frac{23}{12} \\ \frac{25}{12}
	\end{pmatrix}.
\end{equation} 
Thus, we obtain
\begin{equation}
	C
	\begin{pmatrix}
		w_0 \\  w_1 \\ w_2 \\ w_3
	\end{pmatrix}
	= 
	\begin{pmatrix}
		-\frac{5}{4} \\  \frac{13}{12} \\ -\frac{23}{12} \\ \frac{25}{12}
	\end{pmatrix}.
\end{equation} 
This result is equivalent to Eq. \eqref{w0w1w2w3}.

For scheme P2
\begin{equation}
	A=\begin{bmatrix}
		a_{00}&a_{01}&a_{02}&a_{03}&\cdots&0&0&0&0\\
		a_{10}&a_{11}&a_{12}&a_{13}&0&\cdots&0&0&0\\
		0&\frac{1}{6}&\frac{2}{3}&\frac{1}{6}&0&0&\cdots&0&0\\
		0&0&\frac{1}{6}&\frac{2}{3}&\frac{1}{6}&0&0&\cdots&0\\
		\vdots&\ddots&\ddots&\ddots&\ddots&\ddots&\ddots&\ddots&\vdots\\
		0&\cdots&0&0&\frac{1}{6}&\frac{2}{3}&\frac{1}{6}&0&0\\
		0&0&\cdots&0&0&\frac{1}{6}&\frac{2}{3}&\frac{1}{6}&0\\
		0&0&0&\cdots&0&a_{13}&a_{12}&a_{11}&a_{10}\\
		0&0&0&0&\cdots&a_{03}&a_{02}&a_{01}&a_{00}
	\end{bmatrix}
\end{equation}
and
\begin{equation}
	B=\begin{bmatrix}
		b_{00}&b_{01}&b_{02}&b_{03}&\cdots&0&0&0&0\\
		b_{10}&b_{11}&b_{12}&b_{13}&0&\cdots&0&0&0\\
		0&-\frac{1}{2}&0&\frac{1}{2}&0&0&\cdots&0&0\\
		0&0&-\frac{1}{2}&0&\frac{1}{2}&0&0&\cdots&0\\
		\vdots&\ddots&\ddots&\ddots&\ddots&\ddots&\ddots&\ddots&\vdots\\
		0&\cdots&0&0&-\frac{1}{2}&0&\frac{1}{2}&0&0\\
		0&0&\cdots&0&0&-\frac{1}{2}&0&\frac{1}{2}&0\\
		0&0&0&\cdots&0&-b_{13}&-b_{12}&-b_{11}&-b_{10}\\
		0&0&0&0&\cdots&-b_{03}&-b_{02}&-b_{01}&-b_{00}
	\end{bmatrix}.
\end{equation}
In addition, according to Eq. \eqref{wp_relation_w}, 
\begin{equation}
	\begin{pmatrix} 
		\vec{a}_0 & \vec{a}_1  
	\end{pmatrix} 
	\begin{pmatrix}
		w_0'\\w_1'
	\end{pmatrix}
	= \begin{pmatrix}
		w_0 \\  w_1-\frac{1}{6} \\ w_2-\frac{5}{6} \\ w_3- 1
	\end{pmatrix},
\end{equation}
\begin{equation}\label{bsh2-matrix-form-P2} 
	\begin{pmatrix} 
		\vec{b}_0 & \vec{b}_1
	\end{pmatrix} 
	\begin{pmatrix}
		w_0'\\ w_1'
	\end{pmatrix}
	= 
	\begin{pmatrix}
		-1 \\  \frac{1}{2} \\ \frac{1}{2} \\ 0
	\end{pmatrix}.
\end{equation} 
The other weights can be set $w_i'=1,i=2,\dots,N-2$ and $w_i=1,i=4,\dots,N-4$. Using Eq. \eqref{relation_Ai_Bi}, and noticing Eq. \eqref{coe-mat-relation}, one rewrites Eq. \eqref{bsh2-matrix-form-P2} as
\begin{equation}
	C
	\begin{pmatrix}
		w_0 \\  w_1-\frac{1}{6} \\ w_2-\frac{5}{6} \\ w_3- 1
	\end{pmatrix}
	= 
	\begin{pmatrix}
		-1 \\  \frac{1}{2} \\ \frac{1}{2} \\ 0
	\end{pmatrix}.
\end{equation} 
Thus, we obtain
\begin{equation}
	C
	\begin{pmatrix}
		w_0 \\  w_1 \\ w_2 \\ w_3
	\end{pmatrix}
	= 
	\begin{pmatrix}
		-\frac{5}{4} \\  \frac{13}{12} \\ -\frac{23}{12} \\ \frac{25}{12}
	\end{pmatrix}.
\end{equation} 
This result is equivalent to Eq. \eqref{w0w1w2w3}.

For scheme P3, the matrices $A$ and $B$ are of the form Eq. \eqref{matrix_A} and Eq. \eqref{matrix_B}. In addition, according to Eq. \eqref{wp_relation_w}, 
\begin{equation}\label{bsh1-matrix-form}
	\begin{pmatrix} 
		\vec{a}_0 & \vec{a}_1 & \vec{a}_2 
	\end{pmatrix} 
	\begin{pmatrix}
		w_0'\\
		w_1' \\
		w_2'
	\end{pmatrix}
	= \begin{pmatrix}
		w_0 \\  w_1 \\ w_2-\frac{1}{6} \\ w_3- \frac{5}{6}
	\end{pmatrix},
\end{equation}
\begin{equation}\label{bsh2-matrix-form}
	\begin{pmatrix} 
		\vec{b}_0 & \vec{b}_1 & \vec{b}_2 
	\end{pmatrix} 
	\begin{pmatrix}
		w_0'\\
		w_1' \\
		w_2'
	\end{pmatrix}
	= 
	\begin{pmatrix}
		-1 \\ 0 \\ \frac{1}{2} \\ \frac{1}{2}
	\end{pmatrix}.
\end{equation} 
Using Eq. \eqref{relation_Ai_Bi}, and noticing Eq. \eqref{coe-mat-relation}, one rewrites Eq. \eqref{bsh2-matrix-form} as
\begin{equation}\label{bsh1-matrix-form4}
	C 
	\begin{pmatrix}
		w_0 \\  w_1 \\ w_2-\frac{1}{6} \\ w_3- \frac{5}{6}
	\end{pmatrix}
	= \begin{pmatrix}
		-1 \\ 0 \\ \frac{1}{2} \\ \frac{1}{2}
	\end{pmatrix}.
\end{equation}
This result is equivalent to Eq. \eqref{w0w1w2w3}.

\end{proof}

\subsection{The parameter dependencies for schemes P1, P2, and P3}\label{app:dependency}
To optimize the parameters, we need more specific parameter dependencies. 
For scheme P1, we have the following relationship
\begin{equation}\label{P1_relationship}
	\begin{aligned}
		b_{02} &= 0,\\
		b_{03} &= 0,\\
		w_0' &= -\frac{\frac{4}{3}-8w_0}{8a_{00}},\\
		b_{00} &= -\frac{1}{2w_0'},\\
		a_{03} &= -a_{00} - \frac{5}{12}b_{00},\\
		a_{01} &= -5a_{03} - 4b_{00} - 8a_{00},\\
		a_{02} &= -8a_{03} - 2b_{00} - 5a_{00},
		\\
		b_{01} &= -b_{00}.
	\end{aligned}
\end{equation}
There are two free parameters in Eq. \eqref{P1_relationship}. 
In the optimization process, we set $a_{00} = 1$ and $w_0\in (0,10)$.

For scheme P2, we have the following relationship
\begin{equation}\label{P2_w0p_w1p}
	\begin{aligned}
		( b_{00} - \frac{9}{4}a_{03} + b_{03} ) w_0' + ( -\frac{3}{4}a_{11}-\frac{1}{2}b_{11} )w_1' &= \frac{5}{4} -\frac{9}{4}w_3, \\
		( 12a_{00} + 12a_{03} + 4b_{00}-5 b_{03} )w_0' + b_{11}w_1' &= \frac{1}{2},
	\end{aligned}
\end{equation}
\begin{equation}\label{P2_relationship}
	\begin{aligned}
		a_{13} &= \frac{w_3 - 1 - w_0' a_{03}}{w'_1},\\
		b_{13} &= -\frac{w_0'b_{03}}{w_1'},\\
		b_{01} & =12a_{00}+12a_{03}+4b_{00}-5b_{03},\\
		b_{02} & =4b_{03}-12a_{03}-5b_{00}-12a_{00},\\
		a_{01} & =2b_{03}-5a_{03}-4b_{00}-8a_{00},\\
		a_{02} & =4b_{03}-8a_{03}-2b_{00}-5a_{00},\\
		b_{10} & =\frac{9}{4}a_{13} - \frac{3}{4}a_{11} - \frac{1}{2}b_{11} - b_{13},\\
		b_{12} & =\frac{3}{4}a_{11} - \frac{9}{4}a_{13} - \frac{1}{2}b_{11},\\
		a_{10} & =\frac{1}{4}a_{11} - \frac{7}{4}a_{13} + \frac{1}{4}b_{11}+\frac{3}{4}b_{13},\\
		a_{12} & =\frac{1}{4}a_{11} - \frac{15}{4}a_{13} - \frac{1}{4}b_{11} + \frac{9}{4}b_{13}.
	\end{aligned}
\end{equation}
We choose $a_{03}, b_{03}\in(-10,10), w_0\in(0,10)$ as free parameters, and $a_{00} = a_{11} = 1, b_{00} = b_{11} = 0$ are fixed parameters. 

For scheme P3, we have the following relationship
\begin{equation}\label{P3_w1p_w2p}
	\begin{aligned}
		w'_1(\frac{1}{4}a_{11}+\frac{1}{4}b_{11}+\frac{57}{4}a_{13}-\frac{45}{4}b_{13})+w'_2(2b_{22}-5a_{22})&=w_0+16w_3-\frac{58}{3} - w'_0(a_{00}+16a_{03}-12b_{03}),\\
		w'_1(\frac{1}{4}a_{11}-\frac{15}{4}a_{13}-\frac{1}{4}b_{11}+\frac{9}{4}b_{13})+w'_2a_{22}&=w_2-\frac{1}{6}-w'_0(4b_{03}-8a_{03}-2b_{00}-5a_{00}),
	\end{aligned}
\end{equation}
\begin{equation}\label{P3_relationship}
	\begin{aligned}
		a_{23} &= \frac{w_3 - \frac{5}{6}-w_0'a_{03}-w_1'a_{13}}{w_2'},\\
		b_{23} &= \frac{\frac{1}{2}-w_0'b_{03}-w_1'b_{13}}{w_2'},\\
		b_{01} & =12a_{00}+12a_{03}+4b_{00}-5b_{03},\\
		b_{02} & =4b_{03}-12a_{03}-5b_{00}-12a_{00},\\
		a_{01} & =2b_{03}-5a_{03}-4b_{00}-8a_{00},\\
		a_{02} & =4b_{03}-8a_{03}-2b_{00}-5a_{00},\\
		b_{10} & =\frac{9}{4}a_{13} - \frac{3}{4}a_{11} - \frac{1}{2}b_{11} - b_{13},\\
		b_{12} & =\frac{3}{4}a_{11} - \frac{9}{4}a_{13} - \frac{1}{2}b_{11},\\
		a_{10} & =\frac{1}{4}a_{11} - \frac{7}{4}a_{13} + \frac{1}{4}b_{11}+\frac{3}{4}b_{13},\\
		a_{12} & =\frac{1}{4}a_{11} - \frac{15}{4}a_{13} - \frac{1}{4}b_{11} + \frac{9}{4}b_{13},\\
		b_{20} & =12a_{22} + 36a_{23} - 5b_{22} - 28b_{23},\\ 
		b_{21} & =4b_{22} - 36a_{23} - 12a_{22} + 27b_{23}, \\
		a_{20} & =2b_{22} - 16a_{23} - 5a_{22} + 12b_{23},\\
		a_{21} & =4b_{22} - 21a_{23} - 8a_{22} + 18b_{23}.
	\end{aligned}
\end{equation}
We choose $a_{03}, b_{03}, a_{13}, b_{13},  w_0'\in(-10,10), w_0\in(0,10)$ as free parameters, and $a_{00} = a_{11} = a_{22} = 1, b_{00} = b_{11} = b_{22} = 0$ are fixed parameters. 


\subsection{Algorithm for the construction of $\omega_f$}\label{app:algorithm}

Here, we present the algorithm for calculating the resolution of each scheme P1, P2 and P3 under given parameters, which is also the objective function of the optimization problem.

\begin{algorithm}
	\renewcommand{\algorithmicrequire}{\textbf{Input:}}
	\renewcommand{\algorithmicensure}{\textbf{Output:}}
	\caption{Construction of the objective function for the conservative compact scheme}
	\begin{algorithmic}[1]
		\Require{Given parameters $\sigma_0=0.003, \sigma_1=0.002, \sigma_2=0.001, N=101$. Scheme identifier $S \in \{P1, P2, P3\}$, For P1, given $I = 0$, $a_{00} = 1, w_0\in(0,10)$. For P2, given $I = 1$, $a_{03},b_{03}\in(-10,10),w_0\in(0,10),a_{00} = a_{11} = 1, b_{00} = b_{11} = 0$. For P3, given $I = 2$, $a_{03},b_{03},a_{13},b_{13},w_0'\in(-10,10),w_0\in(0,10),a_{00} = a_{11} = a_{22} = 1, b_{00} = b_{11} = b_{22} = 0$
		}
		\Ensure{Average resolution $-\omega_f$}
		\State Compute $w_1, w_2, w_3$ by Eq. \eqref{w0w1w2w3} 
		\If{$w_i \leq 0$ for any $i=1,2,3$}
		\State Return $0$ 
		\EndIf
		\State Compute scheme-specific parameters
		\If{$S = P1$}
		\State Compute $w_0',b_{00}, a_{03},a_{01},a_{02},b_{01}$ by Eq. \eqref{P1_relationship}
		\EndIf
		\If{$S = P2$}
		\If{Linear system Eq. \eqref{P2_w0p_w1p} is unsolvable}
		\State Return 0
		\EndIf
		\State Solve linear system Eq. \eqref{P2_w0p_w1p} for $w_0',w_1'$
		\State Compute $a_{13},b_{13},b_{01},b_{02},a_{01},a_{02},b_{10},b_{12},a_{10},a_{12}$ by Eq. \eqref{P2_relationship}   
		\EndIf 
		\If{$S = P3$}
		\If{Linear system Eq. \eqref{P3_w1p_w2p} is unsolvable}
		\State Return 0
		\EndIf
		\State Solve linear system Eq. \eqref{P3_w1p_w2p} for $w_1',w_2'$
		\State Compute $a_{23}, b_{23},b_{1},b_{02},a_{01},a_{02},b_{10},b_{12},a_{10},a_{12},b_{20},b_{21},a_{20},a_{21}$ by Eq. \eqref{P3_relationship}
		\EndIf 
		\State Construct matrices $A$ and $B$ by Eq. \eqref{matrix_A} and Eq. \eqref{matrix_B}
		\If{$A$ is singular}
		\State Return $0$ 
		\EndIf
		\State Compute $D = A^{-1}B$, $Q = D(2:N, 2:N)$
		\If{real part of largest eigenvalue of $Q > 0$}
		\State Return $0$ 
		\EndIf
		\State Initialize $\omega$-grid: $\omega_{\text{grid}} = [\delta, 2\delta, \ldots, \pi]$, $\delta = 0.01$
		\For{each $i = 0,1,\dots ,I$}
		\State Compute $\bar{\omega}(\omega_{\text{grid}})$ using parameters from $S$ 
		\State Compute $\varepsilon_R(\omega_{\text{grid}}) = \left|\frac{\text{Re}[\bar{\omega}(\omega_{\text{grid}})] - \omega_{\text{grid}}}{\omega_{\text{grid}}}\right|$
		\State Compute $\varepsilon_I(\omega_{\text{grid}}) = \left|\frac{\text{Im}[\bar{\omega}(\omega_{\text{grid}})]}{\omega_{\text{grid}}}\right|$
		\State Find $\omega_{R}^{\sigma_i} = \min\{\omega \in \omega_{\text{grid}} \mid \varepsilon_R(\omega) \geq \sigma_i\}$ 
		and  $\omega_{I}^{\sigma_i} = \min\{\omega \in \omega_{\text{grid}} \mid \varepsilon_I(\omega) \geq \sigma_i\}$
		\If{$\omega_{R}^{\sigma_i}$ or $\omega_{I}^{\sigma_i}$ not found}
		\State \quad Return $0$ 
		\EndIf
		\State Compute $\omega_i^{\sigma_i} = \frac{1}{2}(\omega_{R}^{\sigma_i} + \omega_{I}^{\sigma_i})$
		\EndFor
		\State Compute $\omega_f = \frac{1}{I+1}\sum_{i=0}^I \omega_i^{\sigma_i}$
		\State Return $-\omega_f$ 
	\end{algorithmic}
	\label{alg:Conservative_compact_scheme_objective_function_construction}
\end{algorithm}


\begin{thebibliography}{99}
	


\bibitem{LiMing1995compact}
M. Li, T. Tang, B. Fornberg,
\newblock A compact fourth-order finite difference scheme for the steady incompressible Navier-Stokes equations, International Journal for Numerical Methods in Fluids, 20 (1995), 1137--1151.

\bibitem{Gamet1999compact}
L. Gamet, F. Ducros, F. Nicoud, T. Poinsot,
\newblock Compact finite difference schemes on non-uniform meshes. Application to direct numerical simulations of compressible flows, International Journal for Numerical Methods in Fluids, 29 (1999), 159--191.

\bibitem{Shah2010upwind}
A. Shah, L. Yuan, A. Khan, 
\newblock Upwind compact finite difference scheme for time-accurate solution of the incompressible Navier--Stokes equations, Applied Mathematics and Computation, 215 (2010), 3201--3213.

\bibitem{TianZhenfu2011higher}
Z. Tian, X. Liang, P. Yu,
\newblock A higher order compact finite difference algorithm for solving the incompressible Navier--Stokes equations, International Journal for Numerical Methods in Engineering, 88 (2011), 511--532.

\bibitem{ZhengFeng2021high}
F. Zheng, C.-W. Shu, J. Qiu
\newblock A high order conservative finite difference scheme for compressible two-medium flows, Journal of Computational Physics, 445 (2021), 110597.

\bibitem{FengQiwei2021sixth}
Q. Feng, B. Han, P. Minev,
\newblock Sixth order compact finite difference schemes for Poisson interface problems with singular sources, Computers \& Mathematics with Applications, 99 (2021), 2--25.

\bibitem{FengQiwei2022high}
Q. Feng, B. Han, P. Minev, 
\newblock A high order compact finite difference scheme for elliptic interface problems with discontinuous and high-contrast coefficients, Applied Mathematics and Computation, 431 (2022), 127314.

\bibitem{Tangman2008numerical}
D. Y. Tangman, A. Gopaul, M. Bhuruth,
\newblock Numerical pricing of options using high-order compact finite difference schemes,
Journal of Computational and Applied Mathematics, 218 (2008), 270--280.

\bibitem{Roul2021compact}
P. Roul, V. P. Goura,
\newblock A compact finite difference scheme for fractional Black-Scholes option pricing model, Applied Numerical Mathematics, 166 (2021), 40--60.

\bibitem{CuiMingrong2009compact}
M. Cui,
\newblock Compact finite difference method for the fractional diffusion equation, Journal of Computational Physics, 228 (2009), 7792--7804.

\bibitem{LELE199216}
S. K. Lele,  
\newblock Compact finite difference schemes with spectral-like resolution, Journal of Computational Physics, 103 (1992), 16--42.

\bibitem{kim1997implementation}
J. W. Kim, D. J. Lee, 
\newblock Implementation of boundary conditions for optimized high-order compact schemes, Journal of Computational Acoustics, 5 (1997), 177--191.

\bibitem{wu2024towards}
L. Wu, J. W. Kim, 
\newblock Towards a genuinely stable boundary closure for pentadiagonal compact finite difference schemes, Journal of Computational Physics, (2024), 112887.

\bibitem{qin2021role}
J. Qin, Y. Chen, X. Deng, 
\newblock On the role of global conservation property for finite difference schemes, Journal of Computational Physics, 440 (2021), 110437.

\bibitem{kim2007optimised}
J. W. Kim, 
\newblock Optimised boundary compact finite difference schemes for computational aeroacoustics, Journal of Computational Physics, 225 (2007), 995--1019.

\bibitem{brady2019high}
P. T. Brady, D. Livescu,
\newblock High-order, stable, and conservative boundary schemes for central and compact finite differences, Computers \& Fluids, 183 (2019), 84--101.

\bibitem{gustafsson1975convergence}
B. Gustafsson, 
\newblock The convergence rate for difference approximations to mixed initial boundary value problems, Mathematics of Computation, 29 (1975), 396--406.

\bibitem{strang2000linear}
G. Strang,  
\newblock {Linear algebra and its applications, 2000}.

\bibitem{storn1997differential}
R. Storn, K. Price,
\newblock Differential evolution--a simple and efficient heuristic for global optimization over continuous spaces, Journal of global optimization, 11 (1997), 341--359.
	
	
\end{thebibliography}
\end{document}